\documentclass{amsart}
\usepackage{amsmath,amsfonts,amssymb,amsthm}
\usepackage{graphicx}
\usepackage[a4paper, scale=0.76]{geometry}
\newtheorem{thm}{Theorem}[section]
\newtheorem{lem}[thm]{Lemma}
\newtheorem{prop}[thm]{Proposition}
\newtheorem{cor}[thm]{Corollary}
\theoremstyle{definition}
\newtheorem{defn}[thm]{Definition}
\theoremstyle{remark}
\newtheorem{rem}{Remark}[section]
\newtheorem{exam}[rem]{Example}
\numberwithin{equation}{section}

\def\R{\mathbb R}

\def\N{\mathbb N}
\def\d{\mathrm d}
\DeclareMathOperator{\diag}{diag}

\DeclareMathOperator{\supp}{supp}
\def\D{\mathrm D}
\def\i{\mathrm i}
\def\e{\mathrm e}

\begin{document}

\title{On a wave equation with singular dissipation}
\author{Mohammed Elamine Sebih}
\address{Mohammed Sebih, Laboratory of Analysis and Control of Partial Differential Equations, {Djillali Liabes University}, {22000 Sidi Bel Abbes, {Algeria}}}
\email{sebihmed@@gmail.com}
\author{Jens Wirth*}
\address{Jens Wirth, { Department of Mathematics}, {University of Stuttgart}, { Pfaffenwaldring 57, 70569 Stuttgart, {Germany}}}
\email{jens.wirth@@mathematik.uni-stuttgart.de}



\begin{abstract}
In this paper we consider a singular wave equation with distributional and more singular non-distributional coefficients and develop tools and techniques for the phase-space analysis of such problems. In particular we provide a detailed analysis for the interaction of singularities of solutions with strong singularities of the coefficient in a model problem of recent interest. 
\end{abstract}



\maketitle


\section{Introduction}

In a recent paper, Munoz, Ruzhansky and Tokmagambetov \cite{MRT17} investigated a particular wave model with singular dissipation arising from acoustic problems. They considered the Cauchy problem
\begin{equation*}
\begin{split}
u_{tt}-\Delta u+\frac{b^{\prime }(t)}{b(t)}u_{t}=0, \qquad
u(0,x)=u_{0}(x), \quad u_{t}(0,x)=u_{1}(x),
\end{split}
\end{equation*}
where $b$ is a piecewise continuous and positive function allowing in particular for jumps and in consequence a non-distributional singular coefficient in this Cauchy problem. They considered the notion of very weak solutions for this singular problem and showed that this problem is well-posed in this very weak sense. Moreover, they numerically observed in one space dimension a very interesting phenomenon, namely the appearance of a new wave after the singular time travelling in the opposite direction to the main one.

The aim of this paper is two-fold. On one hand we consider this model and carry out a detailed phase space analysis for families of regularised problems in order to describe the behaviour of the very weak solution in the vicinity of the singular time. This will allow us to show that the numerically observed partial reflection of wave packets at the singular time is really appearing and to calculate the partial reflection indices in terms of the jump of the coefficient. On the other hand this is a model study to develop tools and techniques to treat more general singular hyperbolic problems within the framework of very weak solutions and to provide a symbolic calculus framework for analysing singularities of such solutions. 

\section{The notion of very weak solutions}
We will recall some basic concepts on the notion of very weak solutions for singular problems and comment on their relation to other solution concepts like weak solutions and Colombeau solutions. The concept was introduced by Garetto and Ruzhansky in \cite{GaRu} and further developed in a series of papers with different co-authors, \cite{RT17}, \cite{RT17b}, \cite{MRT17}, \cite{RT18}, \cite{MRT19}  in order to show a wide applicability. The basic idea is as follows. Instead of considering the singular equation itself, one considers a family of regularised equations depending on a regularisation parameter and investigates the behaviour of the family of solutions as the regularisation parameter tends to zero. 

Treating distributions and more singular objects as families of regularised objects has a long history. In order to provide a neat solution for the multiplication problem for distributions (on the background of Schwartz's famous impossibility result in  \cite{Schwartz54}) Colombeau \cite{Col} proposed to consider more general algebras of nets of regularised objects modulo negligible nets 
\begin{equation} \mathcal E^\infty(\Omega)  / \mathcal N^\infty(\Omega), \end{equation}
where $\mathcal E^\infty(\Omega)$ denotes all functions $ (0,1] \ni\epsilon\mapsto f_\epsilon\in \mathrm C^\infty(\Omega)$ being moderate in the sense that 
\begin{equation}\label{eq:2.2}
   \sup_{x\in K} |\partial^\alpha f_\epsilon(x) | = \mathcal O(\epsilon^{N-|\alpha|})
\end{equation}
for some $N\in \mathbb R$ depending on $K\Subset \Omega$ and the multi-index $\alpha$ and similarly $\mathcal N^\infty(\Omega)$ the space of negligible nets satisfying the estimate \eqref{eq:2.2} for any number $N\in\R$. Convolution with Friedrichs mollifiers yields an embedding of both smooth functions $\mathrm C^\infty(\Omega)$ and distributions $\mathcal D'(\Omega)$ into this algebra extending in particular multiplication of smooth functions. For more details, see Oberguggenberger \cite{Oberguggenberger}.

This approach has a serious drawback as the multiplication in these algebras is only consistent with the multiplication of smooth functions, and hence, in general not consistent with the algebra structure of continuous or measurable functions. This is in particular problematic when applying this concept to well-posedness issues of singular partial differential equations, where the natural spaces are usually of lower regularity than $\mathrm C^\infty$. To overcome consistency issues, in \cite{GaRu} and later in \cite{RT17}, \cite{RT17b} Ruzhansky and his co-authors introduced a different concept of moderateness and negligibility based on natural norms associated to the problem under consideration. 

For hyperbolic partial differential equations it seems natural to consider solutions of finite energy and the modification in the approach would be to call a family of solutions moderate if the energy satisfies a polynomial bound with respect to the regularisation parameter, while negligible nets are such that the energy is smaller than any power of the regularisation parameter. Thus the notion of very weak solutions depends on the equation under consideration (in contrast to distributional and Colombeau solutions, but similar to weak or mild solutions). 

To make this precise, let us define the notion of very weak solutions for wave equations with singular time-dependent coefficients
\begin{equation}
   \partial_t^2 u - a(t) \Delta u + 2b(t) \partial_t u + m^2(t) u = f
\end{equation}
for a given singular right-hand side $f$ and singular coefficients $a,b,m$.  We say that a net $\epsilon\mapsto u_\epsilon
\in \mathrm C^\infty([0,T] ; \mathrm H^1)$ is a \emph{very weak solution of (finite) energy type}, if there are moderate regularisations 
$a_\epsilon,b_\epsilon,m_\epsilon \in \mathcal E^\infty([0,T])$ of the coefficients and a $\mathrm C^\infty([0,T]; \mathrm L^2)$-moderate regularisation of the right-hand side $f$ in the sense that
\begin{equation}
   \sup_{t\in [0,T]} \|\partial^k_t f_\epsilon(t,\cdot) \|_{\mathrm L^2} = \mathcal O(\epsilon^{N-k})
\end{equation}
for some number $N\in\R$, such that $ \partial_t^2 u_\epsilon - a_\epsilon(t) \Delta u_\epsilon + 2b_\epsilon(t) \partial_t u_\epsilon + m_\epsilon^2(t) u_\epsilon = f_\epsilon$ holds for any $\epsilon>0$ and $u_\epsilon$ itself is $\mathrm C^\infty([0,T]; \mathrm H^1)$-moderate in the sense that
\begin{equation}
  \sup_{t\in [0,T]} \|\partial^k_t u_\epsilon(t,\cdot) \|_{\mathrm H^1} = \mathcal O(\epsilon^{N-k})
\end{equation}
holds for some $N\in\R$.

Based on results from \cite{W6} and \cite{W5}, it was shown in \cite{MRT17} that the model example we will consider later is well-posed in this very weak sense and that the very weak solution is independent of the choice of the regularising family. For the general singular wave model with singular speed and mass term see \cite{ART19}.

\section{Our model problem and general strategy}
We consider the Cauchy problem
\begin{equation}
\begin{split}\label{Model}
u_{tt}-\Delta u+\frac{b^{\prime }(t)}{b(t)}u_{t}=0 ,\qquad u(0,x)=u_{0}(x),\quad u_{t}(0,x)=u_{1}(x),
\end{split}
\end{equation}
where $b$ is a piecewise smooth and piecewise continuous function. We are interested in solutions close to a singularity of the coefficient and hence, without loss of generality, we assume that $b$ has exactly one jump at $t=1$. In particular, we
require that the limits 
 \begin{equation}b(1_{\pm0})=\lim_{t\rightarrow 1\pm0} b(t)\end{equation}
exist for the function itself and also its derivatives. Thus, we ask for $b$ to satisfy the following two assumptions:
\begin{description}
\item[(H1)]   There exists a strictly positive number $b_{0}$ such that $b(t)\geq b_{0}>0$.
\item[(H2)]   $b\in C^{\infty}_b ( -\infty,1]  \cap C^{\infty}_b [1,+\infty)$, having a jump at $t=1$.
\end{description}
In contrast to \cite{MRT17} we do \emph{not} require $b$ to be monotonically increasing. Thus, we will not make use of any sign properties of the coefficient later.

\subsection{Notation}\label{sec:Notation}
Throughout this paper we will use the following conventions and symbols:
\begin{itemize}
\item
 We denote the height of the jump of $b$ at $t=1$ by $h=b(1_{+0})-b(1_{-0})$ and denote $H = \frac{b(1_{-0})}{b(1_{+1})}$.
 \item
We write $f\lesssim g$ for two functions $f$ and $g$ on the same domain if there exists a positive constant $C$ such that $f\leq Cg$.
\item 
We denote by $\Phi \in \mathrm{C}_{0}(\mathbb{R})$ a fixed non-negative, continuous and symmetric function, such that 
\begin{equation}
   \Phi(-t)=\Phi(t)\qquad\text{and}\qquad \supp\Phi=[-K',K']
\end{equation}
holds. We further assume that $\Phi$ is differentiable outside the origin and that
\begin{equation}\label{eq:PhiSquEst1}
    \Phi^2(t) \lesssim \begin{cases}  \Phi'(t) ,\qquad & t<0,\\ -\Phi'(t) , & t>0, \end{cases}
\end{equation}
holds.
This function will play an important role in the definition of zones and symbol classes and will be referred to as the shape function.
\item
 We denote by $\psi\in \mathrm{C}_{0}^{\infty }(\mathbb{R})$ a fixed non-negative and symmetric mollifier such that 
 \begin{equation} 
 \psi(-t)=\psi(t),\qquad \supp\psi=[-K,K], \qquad\text{ and }\qquad \int \psi(t)\,\mathrm{d}t=1
 \end{equation}
 with $0<K\le K'$ describing the size of its support. We further require that derivatives of $\psi$ be bounded by powers of the shape function $\Phi$, i.e.
 \begin{equation}
     |\partial^k_t \psi(t) | \lesssim \Phi^k(t)
 \end{equation}
 for any number $k\in\N$.
\item The identity matrix will be denoted by $\mathrm I$. Furthermore for any matrix $A$ we denote by $\|A\|$ its Euclidean matrix norm. 
\end{itemize}

\subsection{Regularisation of the problem}
In order to consider very weak solutions of our model problem, we solve families of regularised problems using the regularisations 
\begin{equation}
\begin{split}\label{beps}
b_{\epsilon }(t)=b\ast \psi _{\epsilon }(t) \qquad\text{and}\qquad
b_{\epsilon }^{\prime }(t)=b^{\prime }\ast \psi _{\epsilon }(t) = b \ast \psi_\epsilon^\prime(t)
\end{split}
\end{equation}
in terms of the mollifier $\psi _{\epsilon }(t)=\epsilon ^{-1}\psi ( \epsilon ^{-1}t) 
$ and with $\epsilon \in \left( 0,1\right]$. This gives rise to the family of Cauchy problems
\begin{equation}
\begin{split}
u_{tt}-\Delta u+\frac{b_{\epsilon }^{\prime }(t)}{b_{\epsilon }(t)}u_{t}=0,\qquad
u(0,x)=u_{0}(x),\quad u_{t}(0,x)=u_{1}(x)
\end{split}
\label{Regularised pb}
\end{equation}
parameterised by $\epsilon\in(0,1]$. Our approach is based on a detailed phase space analysis for this family of problems treating $\epsilon$ as an additional variable of the extended phase space. For this, we will introduce two zones and apply a diagonalisation based technique to extract leading order terms in each of them. For details on the diagonalisation procedure and its use in a related singular context we refer to \cite{NW15} or \cite{W16}, and for a broader discussion of the techniques used see \cite{RW14}. 

As the coefficients of \eqref{Regularised pb} depend on $t$ only, we apply a partial Fourier transform with respect to the spatial variables and, thus, reduce consideration to the ordinary differential equation
\begin{equation}
\begin{split}
\hat{u}_{tt}+\vert \xi \vert ^{2} \hat{u}+\frac{b_{\epsilon }^{\prime }(t)}{b_{\epsilon }(t)}\hat{u}_{t}=0,\qquad
\hat{u}(0,\xi)=\hat{u}_{0}(\xi),\quad \hat{u}_{t}(0,\xi)=\hat{u}_{1}(\xi)
\label{Eqn in phase space}
\end{split}
\end{equation}
parameterised by both $\epsilon\in(0,1]$ and $\xi\in\R^n$. We construct its solutions for $t\in[0,2]$ and investigate the limiting behaviour of solutions as $\epsilon\to0$. To write the equation in system form, we introduce the micro-energy
\begin{equation}
U(t,\xi,\epsilon )=\begin{pmatrix} \left\vert \xi \right\vert \hat{u} \\ \D_{t}\hat{u}\end{pmatrix},
\end{equation}
where $\D_{t}=-\i\partial _{t}$ denotes the Fourier derivative. Then (\ref{Eqn in phase space}) can be rewritten as
\begin{equation}
\D_{t}U(t,\xi,\epsilon )=\left[ \begin{pmatrix}
0 & \vert \xi \vert \\ 
\vert \xi \vert & 0
\end{pmatrix}+\begin{pmatrix}
0 & 0 \\ 
0 & \i\mathfrak d_{\epsilon}(t)
\end{pmatrix}\right] U(t,\xi,\epsilon ),  \label{Transf to system}
\end{equation}
where we used the notation $\mathfrak d_{\epsilon}(t)=\frac{b_{\epsilon }^{\prime }(t)}{b_{\epsilon }(t)}$ for the net of dissipation coefficients. Denoting the coefficient matrices arising in this system by
\begin{equation}
A(\xi)=\begin{pmatrix} 
0 & \vert \xi \vert \\ 
\vert \xi \vert & 0
\end{pmatrix}\quad \text{  and  }\quad
B(t,\epsilon)=\begin{pmatrix} 
0 & 0 \\ 
0 & \i\mathfrak d_{\epsilon}(t)
\end{pmatrix},
\end{equation}
we see that depending on the values $|\xi|$, $\epsilon$ and $t$ either the matrix $A(\xi)$ or the matrix $B(t,\epsilon)$ is dominant. If $A(\xi)$ is dominant, we apply a standard hyperbolic approach and diagonalise the system. If $B(t,\epsilon)$ is dominant, we use a transformation of variables to reduce consideration to a model equation describing the behaviour close to the singularity. 
 
 \subsection{Zones}
 To make use of different leading terms, we use the following definition of zones. For a zone constant $N$  to be fixed later we define the \emph{hyperbolic zone}
\begin{equation}
Z_{\rm hyp}(N)=\left\{( t,\xi,\epsilon) \in [0,2]\times\mathbb{R}^{n}\times(0,1] \mid |\xi|\geq N\big(\Phi_{\epsilon }(t-1)+1\big)\right\rbrace ,
\end{equation}
where $\Phi_{\epsilon }(t)=\epsilon ^{-1}\Phi \left( \epsilon ^{-1}t\right) $ is defined in terms of the function $\Phi$ from Section \ref{sec:Notation}. The \emph{singular zone}
\begin{equation}
Z_{\rm sing}(N)=\left\{( t,\xi,\epsilon) \in [0,2]\times\mathbb{R}^{n}\times(0,1] \mid  N < |\xi|\leq N\big(\Phi_{\epsilon }(t-1)+1\big)\right\} 
\end{equation}
is used to investigate the vicinity of the jump of the coefficient, while the remaining \emph{bounded frequencies} 
\begin{equation}
Z_{\rm bd}(N)=\left\{( t,\xi,\epsilon) \in [0,2]\times\mathbb{R}^{n}\times(0,1] \mid |\xi|\leq N \right\} 
\end{equation}
will be dealt with later by a simple argument. The common boundary of the hyperbolic and the singular zone will be denoted by 
$(t_{\xi_{i}}(\epsilon))_{i=1,2}$ and is defined implicitly by the equation 
\begin{equation}\label{eq:txi-def}
|\xi| = N\big(\Phi_{\epsilon }(t_{\xi_{i}}-1)+1\big)
\end{equation} 
for $\xi$ satisfying $N<|\xi|\leq N(\epsilon^{-1}\Phi(0)+1)$ and with the convention that $t_{\xi_{1}}$ is the solution branch for $t < 1$ and $t_{\xi_{2}}$ when $t > 1$. The zones are depicted in Figure~\ref{Fig1} for fixed $\epsilon>0$.

\begin{figure}[b]
\centering
\includegraphics[height=6.4cm]{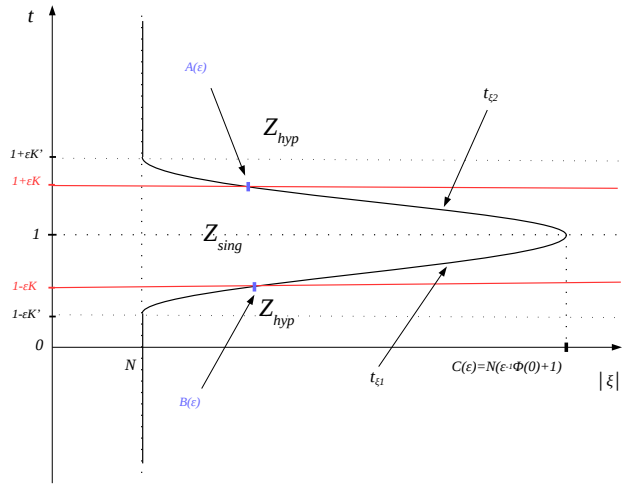}
\caption{Zones in coordinates $(t,\xi)$ for fixed $\epsilon>0$}\label{Fig1}
\end{figure}

The singular zone $Z_{\rm sing}(N)$ is better understood in the variables  $\Lambda = \epsilon|\xi|$ and $\tau = \epsilon^{-1}(t-1)$. Then the definition of the singular zone can be rewritten as
\begin{equation}
Z_{\rm sing}(N)=\left\{( \tau,\Lambda,\epsilon) \mid  N\epsilon < \Lambda \leq N \Phi(\tau)+N\epsilon\right\rbrace
\end{equation}
and stabilises as $\epsilon\to0$. We will use these singular variables when discussing the solutions of the regularised problem in the singular zone. For convenience, the zone is depicted in Figure~\ref{Fig2} using these variables. We will also use a notation for the zone-boundary and denote it by  $\tau_{\Lambda_{1}}(\epsilon)$ and $\tau_{\Lambda_{2}}(\epsilon)$.

\begin{figure}[t]
\centering
\includegraphics[height=6.4cm]{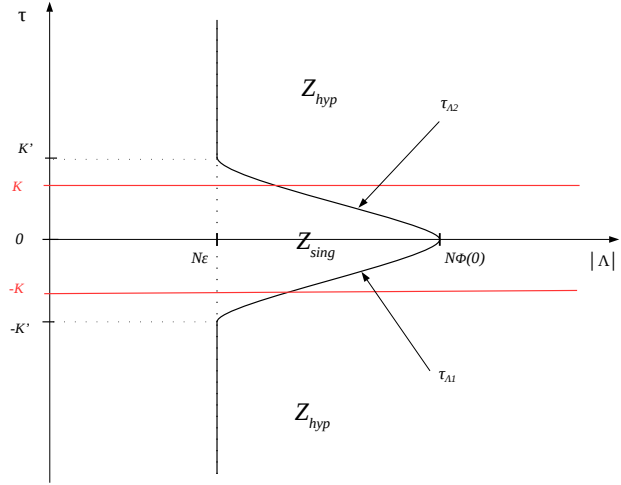}
\caption{Zones in coordinates $(\tau,\Lambda)$ again for a fixed $\epsilon>0$}
\label{Fig2}
\end{figure}

Our strategy is as follows. Within the hyperbolic zone we will apply a diagonalisation procedure taking care of the $\epsilon$-dependence of the transformation matrices and all appearing symbols in an appropriate way. This allows to construct the fundamental solution of the parameter-dependent family \eqref{Regularised pb} within the hyperbolic zone and to investigate its limiting behaviour as $\epsilon\to0$. Within the singular zone, we transform the problem to the singular variables and construct its fundamental solution as power series in $\Lambda$ with $\tau,\epsilon$-dependent coefficients and again study the limiting behaviour of this solution as $\epsilon\to0$.

\begin{rem}
We note that in coordinates $(t,\xi)$ the point $C(\epsilon)$ tends to $\infty$ when $\epsilon$ tends to $0$ and that $t_{\min}$ and $t_{\max}$ depend on $\epsilon$ and tend to $1_{-0}$ and $1_{+0}$ when $\epsilon\rightarrow 0$, respectively.  
\end{rem}

\begin{rem}
The interval $\left[1-\epsilon K ,1+\epsilon K\right]$ is the support of $\psi_{\epsilon}(t-1)$. The lines $t=1-\epsilon K$ and $t=1+\epsilon K$ divide the hyperbolic zone into two parts, one with $\vert t-1\vert > \epsilon K$ and the other one with $\vert t-1\vert < \epsilon K$. The last one is of minor interest since the points $A(\epsilon)=N\epsilon^{-1}\Phi(K)+N$ and $B(\epsilon)$ tend to infinity when $\epsilon$ tends to $0$. 
\end{rem}

\subsection{Regular faces of the zones.} The hyperbolic zone $Z_{\rm hyp}(N)$ and the zone of bounded frequencies $Z_{\rm bd}(N)$ have a boundary on which $\epsilon\to0$. This will be of importance later when relating our representation of very weak solutions to the standard theory for smooth coefficients for $t\ne 1$. 
We refer to the two parts $\{ (t,\xi,0) \mid |\xi|>N,\; t\ne 1\}$ as the \emph{regular face} of $Z_{\rm hyp}(N)$ and 
the set $\{ (t,\xi,0) \mid |\xi|\le N\}$ as the \emph{regular face} of $Z_{\rm bd}(N)$. 
The singular zone does not have a regular face.

\section{Representation of solutions}
\subsection{Some useful lemmata}
The nets $\mathfrak d_{\epsilon}(t)={b_{\epsilon }^{\prime }(t)}/{b_{\epsilon }(t)}$, $b_{\epsilon }(t)$ and its derivatives $b_{\epsilon }^{(k)}(t)$ defined in terms of \eqref{beps} satisfy the following inequalities.
\begin{lem}\label{lem:B_epsilon}
The estimates
\begin{equation}
\begin{split}
|\partial_t^k b_\epsilon(t)| \le C_{1,k} \big(\Phi_\epsilon(t-1)+1\big)^{k}\qquad\text{and}\qquad
\vert\partial_{t}^{k}\mathfrak d_{\epsilon}(t)\vert\leq C_{2,k}  \big( \Phi_{\epsilon}(t-1)+1 \big)^{k+1}
\end{split}
\end{equation}
hold for all $k\geq0$ and all $t\in[0,2]$ and $\epsilon\in(0,1]$.
\end{lem}
\begin{proof}
The second estimate will follow from the first one, so we only concentrate on the first. For $k=0$ we have by that by Assumptions \textbf{(H1)} and \textbf{(H2)} and the positivity of the mollifier $\psi$ 
\begin{equation}
   0 < b_0 \le b_\epsilon(t) = \int_{-\infty}^\infty b(t-\epsilon s) \psi(s) \d s \le \max_{s\in[- K,1+ K]} b(s).
\end{equation}
For $k=1$ we apply integration by parts. As
\begin{equation}
\begin{split}
b_\epsilon'(t) &=  \epsilon^{-2}\int_{-\infty}^{\infty}b(s)\psi^{\prime}(\epsilon^{-1}(t-s))\d s
= 
 \epsilon^{-2}\,\bigg[\int_{-\infty}^{1}b(s)\psi^{\prime}(\epsilon^{-1}(t-s))\d s + \int_{1}^{\infty}b(s)\psi^{\prime}(\epsilon^{-1}(t-s))\d s \bigg] \\
 &= \epsilon^{-1}  \psi(\epsilon^{-1} (t-1)) \big[ b(1_{-0}) - b(1_{+0}) \big]   -  \epsilon^{-1}\,\bigg[\int_{-\infty}^{1}b'(s)\psi(\epsilon^{-1}(t-s))\d s + \int_{1}^{\infty}b'(s)\psi(\epsilon^{-1}(t-s))\d s \bigg] ,
\end{split}
\end{equation}
we obtain 
\begin{equation}
  | b'_\epsilon(t) | \le |h| \psi_\epsilon(t-1) + \sup_{s\in[- K,1)\cup(1,1+ K]} |b'(s)|
  \le C_1 \left( \Phi_\epsilon(t-1) +  1\right)
\end{equation}
using the bound $\psi\lesssim \Phi$. 
For higher $k$ we need to apply several steps of integration by parts. For $k\ge2$ we obtain by induction
\begin{equation}
\begin{split}
   \partial_t^k b_\epsilon(t) &= \sum_{\ell=0}^{k-1} (-1)^{k-\ell-1} \big[b^{(k-\ell-1)}(1_{-0}) - b^{(k-\ell-1)}(1_{+0}) \big]  \partial_t^{\ell}\psi_\epsilon(t-1) \\
  &\quad + (-1)^{k} \bigg[\int_{-\infty}^{1}b^{(k)}(s)\psi_\epsilon(t-s)\d s + \int_{1}^{\infty}b^{(k)}(s)\psi_\epsilon(t-s)\d s \bigg]. 
\end{split}
\end{equation}
Again, the remaining integrals can be estimated by uniform bounds on the derivatives of $b$ outside the singularity,
and the statement follows from
\begin{equation}
\begin{split}
    | \partial_t^k b_\epsilon(t)  | &\le \sum_{\ell=0}^{k-1}  |b^{(k-\ell-1)}(1_{-0}) - b^{(k-\ell-1)}(1_{+0}) |\,
     |\partial_t^\ell \psi_\epsilon (t) |  + \sup_{s\in[- K,1)\cup(1,1+ K]} |b^{(k)}(s)|\\
     &\le C_k \big( \Phi_\epsilon(t)+1 \big)^k
\end{split}
\end{equation}
by the estimate $|\partial_t^k\psi(t)|\le C \Phi^k(t)$.
Finally, the estimate for derivatives of $\mathfrak d_\epsilon(t)$ follow from applying the quotient rule and using the uniform lower bound $b_0\le b_\epsilon(t)$ when estimating the denominator.
\end{proof}

\begin{lem}\label{lem:b_epsilon-b_0}
The estimates
\begin{equation}
\vert \partial_t^k b_{\epsilon}(t)-\partial_t^k b(t)\vert \lesssim \epsilon \label{Estimate b_epsilon-b_0}
\end{equation}
hold for all $k\ge0$ uniformly in $\epsilon\in(0,1]$ and $t$ satisfying $\vert t-1\vert > \epsilon K$.
\end{lem}

\begin{proof}
Let $\vert t-1\vert\ > \epsilon K$. For $k=0$ we have
\begin{equation}
    b_\epsilon(t) - b(t) = \int_{-\infty}^\infty b(s) \psi_\epsilon(t-s)\d s - b(t) = 
    \int_{\left[t-\epsilon K,t+\epsilon K\right]} \left(b(s) - b(t) \right) \psi_\epsilon(t-s)\d s
\end{equation}
seeing that $\int\psi(s)\d s=1$ and that $\supp\psi_\epsilon=[-\epsilon K,\epsilon K]$. Hence,
\begin{equation}
\vert b_{\epsilon}(t)-b(t)\vert  \leq \int_{\left[t-\epsilon K,t+\epsilon K\right]}\vert b(s)-b(t)\vert\psi_{\epsilon}(t-s)
\d s.
\end{equation}
 As the range of integration does not contain $1$ we can use the differentiability of $b$ to estimate 
\begin{equation}
\vert b(s)-b(t)\vert \leq |s-t| \, \sup_{\theta \in [t-\epsilon K,t+\epsilon K]} |b'(\theta)| = M\vert s-t\vert
\end{equation}
for $s\in \left[t-\epsilon K,t+\epsilon K\right]$. Therefore
\begin{align}
\vert b_{\epsilon}(t)-b(t)\vert 
 \leq  \epsilon M K.
\end{align}
For $k\ge 1$ the argumentation is similar using
\begin{align}
\partial_t^k b_{\epsilon}(t)-\partial_t^k b(t) = \int_{\left[t-\epsilon K,t+\epsilon K\right]}\left( \partial_s^k b(s)-\partial_t^k b(t)\right) \psi_{\epsilon}(t-s)\d s
\end{align}
together with the corresponding bound on the derivatives of $b$ on the interval of integration.
\end{proof}

These two technical lemmas are the model behaviours for our symbol classes and the key estimates for the boundary behaviour at regular faces of the zones.

\subsection{Treatment in the hyperbolic zone}
\subsubsection{Symbol classes and their properties} For the treatment within the hyperbolic zone, symbol classes and their basic calculus properties are used. 

\begin{defn}[Symbol classes]
Let $N>0$ be fixed and $\Phi$ as in Section~\ref{sec:Notation}.
\begin{enumerate}
\item[\bf (i)] We say that a function 
\begin{equation}a \in C^{\infty}([0,2]\times\mathbb{R}^{n}\times(0,1])\end{equation} 
belongs to the hyperbolic symbol class $\mathcal S_{N,\Phi}\lbrace m_1,m_2\rbrace$ if it satisfies the estimates
\begin{equation}
\vert \partial_t^k \partial_{\xi}^{\alpha} a(t,\xi,\epsilon)\vert \leq C_{k,\alpha}\big(\Phi_{\epsilon }(t-1)+1\big)^{m_2+k}|\xi|^{m_1-\vert\alpha\vert}
\end{equation}
uniformly within $Z_{\rm hyp}(N)$ for all non-negative integers $k\in\mathbb N_0$ and all multi-indices $\alpha \in \mathbb{N}_0^n$
 together with the existence of the limits 
\begin{equation} a(t,\xi,0)=\lim_{\epsilon\to0} a(t,\xi,\epsilon),\qquad t\ne1, \end{equation}
at the regular face of the zone satisfying the estimates
\begin{align}
   |\xi|^{|\alpha|-m_1}    | \partial_t^k \partial_\xi^\alpha a(t,\xi,0) | &\le C_{k,\alpha}',  \label{eq:symb_RegF}  \\
   |\xi|^{|\alpha|-m_1}     | \partial_t^k\partial_\xi^\alpha \big( a(t,\xi,\epsilon) - a(t,\xi,0) \big) | & \le C_{k,\alpha}'' \epsilon
    \label{eq:limit_RegF}
\end{align}
with the latter one uniformly on $|t-1|\ge \epsilon K$.

\item[\bf (ii)] We say that a matrix-valued function $A$ belongs to $\mathcal S_{N,\Phi}\lbrace m_1,m_2\rbrace$ if all its entries belongs to the scalar-valued symbol class $\mathcal S_{N,\Phi}\lbrace m_1,m_2\rbrace$. 
\end{enumerate}
\end{defn}

\begin{exam}
Due to Lemma~\ref{lem:B_epsilon}, we know that the regularising families $b_\epsilon$ and $\mathfrak d_\epsilon$ satisfy 
\begin{equation}  
(b_\epsilon) \in \mathcal S_{N,\Phi}\{0,0\}\qquad \text{and}\qquad 
(\mathfrak d_{\epsilon})\in \mathcal S_{N,\Phi} \left\lbrace0,1\right\rbrace
\end{equation}
for any zone constant $N>0$. Similarly $|\xi|$ is a symbol from $\mathcal S_{N,\Phi} \{1,0\}$ for any admissible $\Phi$ and $N>0$.  
\end{exam}

\begin{rem}
The boundary behaviour of symbols given by \eqref{eq:symb_RegF} corresponds to a characterisation of symbol classes defined on the regular face $\{(t,\xi,0) \mid |\xi|\ge N\}$ of $Z_{\rm hyp}(N)$ with symbol estimates uniform with respect to $t$.
\end{rem}

Increasing the zone constant $N$ makes the hyperbolic zone $Z_{\rm hyp}(N)$ smaller and thus the symbol class $\mathcal S_{N,\Phi}\{m_1,m_2\}$ larger. We will make use of this fact later by choosing $N$ sufficiently large in order to guarantee the smallness of some terms. We will omit the indices $N$ and $\Phi$ to simplify notation.

\begin{prop}[Properties of symbol classes]\label{prop:Symbolic properties} For any fixed $N>0$ and admissible $\Phi$ the following statements hold:
\begin{enumerate}
\item[\bf(1)] $\mathcal S\lbrace m_1,m_2\rbrace$ is a vector space.
\item[\bf(2)] $\mathcal S\lbrace m_1,m_2\rbrace\subset \mathcal S\lbrace m_1+\ell_1,m_2-\ell_2 \rbrace$ for all $\ell_1 \geq \ell_2 \geq0$. 
\item[\bf(3)] If $f \in \mathcal S\lbrace m_1,m_2\rbrace$ and $g \in\mathcal S\lbrace m_1^\prime,m_2^\prime\rbrace$ then $f\cdot g \in \mathcal S\lbrace m_1+m_1^\prime,m_2+m_2^\prime\rbrace$.
\item[\bf(4)] If $f \in\mathcal S\lbrace m_1,m_2\rbrace$ then $\partial_t^k f \in\mathcal S\lbrace m_1,m_2+k\rbrace$ and $\partial_{\xi}^{\alpha} f \in \mathcal S \lbrace m_1-\vert\alpha\vert,m_2 \rbrace$.
\item[\bf(5)] If $f \in\mathcal S\lbrace m_1,0\rbrace$ satisfies $\vert f(t,\xi,\epsilon)\vert\textgreater c|\xi|^{m_1}$ for a positive constant $c$, then one has ${1}/{f}\in\mathcal S\lbrace -m_1,0\rbrace$. 
\end{enumerate}
\end{prop}

\begin{proof}
Properties (1) and (4) follow immediately from the definition of the symbol classes. For (3) we apply the product rule for derivatives to derive the symbol estimate \eqref{eq:symb_RegF}. The boundary behaviour \eqref{eq:limit_RegF} follows by writing the product of two symbols $f \in \mathcal S\{m_1,m_2\}$ and $g\in\mathcal S\{m_1',m_2'\}$ as 
\begin{equation}
     f(t,\xi,\epsilon) g(t,\xi,\epsilon) - f(t,\xi,0) g(t,\xi,0) 
    =    \big( f(t,\xi,\epsilon) - f(t,\xi,0)\big) g(t,\xi,\epsilon) + f(t,\xi,0)\big( g(t,\xi,\epsilon) - g(t,\xi,0)\big),
\end{equation}
applying the product rule for derivatives 
\begin{equation}
\begin{split}
 \partial_t^k\partial_\xi^\alpha \big(    f(t,\xi,\epsilon) g(t,\xi,\epsilon) - f(t,\xi,0) g(t,\xi,0)  \big) 
 &= \sum_{\substack{\ell\le k\\ \beta\le \alpha}} \binom{k}{\ell}  \binom{\alpha}{\beta}   \big( \partial_t^{k-\ell} \partial_\xi^{\alpha-\beta} \big( f(t,\xi,\epsilon) - f(t,\xi,0)\big) \big)\, \big(\partial_t^\ell \partial_\alpha^\beta g(t,\xi,\epsilon) \big)\\&\qquad\quad
 +\binom{k}{\ell}  \binom{\alpha}{\beta}  \big( \partial_t^{k-\ell} \partial_\xi^{\alpha-\beta} f(t,\xi,0)\big)\, \big(\partial_t^\ell \partial_\alpha^\beta\big( g(t,\xi,\epsilon) - g(t,\xi,0)\big)\big),
 \end{split}
\end{equation}
and estimating each resulted differences on the right by \eqref{eq:limit_RegF} and all the remaining factors by \eqref{eq:symb_RegF}.
To prove (2) we use the definition of the hyperbolic zone $Z_{\rm hyp}(N)$ in the form
\begin{equation}
    \big(\Phi_\epsilon(t-1)+1\big)^{-\ell_2}|\xi|^{\ell_1} \ge N^{\ell_1} \big(\Phi_\epsilon(t-1)+1\big)^{\ell_1-\ell_2} \ge N^{\ell_1}
\end{equation}
and conclude that symbol estimates from $\mathcal S\{m_1,m_2\}$ imply symbol estimates from $\mathcal S\{m_1+\ell_1,m_2-\ell_2\}$. It remains to prove (5). Here we use Fa\`a di Bruno's formula (see \cite{MHFaa}, \cite{DLT}, \cite{DLTerr}) and write
\begin{equation}
\partial_{t}^{k}\partial_{\xi}^{\alpha} \frac{1}{f(t,\xi,\epsilon)}
=\sum_{\ell=1}^{k+\vert\alpha\vert}\sum_{\substack{j_{1}+\cdots+j_{\ell}=k \\ 
\vert\alpha_{1}\vert+\cdots+\vert\alpha_{\ell}\vert=\vert\alpha\vert}
}
C_{k,\alpha,j_{i},\alpha_{i}}\frac{\partial_{t}^{j_1}\partial_{\xi}^{\alpha_1}f(t,\xi,\epsilon)\cdots \partial_{t}^{j_\ell}\partial_{\xi}^{\alpha_\ell}f(t,\xi,\epsilon)}{f^{\ell+1}(t,\xi,\epsilon)}
\end{equation}
where $C_{k,\alpha,j_{i},\alpha_{i}}$ are constants depending on the order of the derivatives. Each term in the last sum can be estimated in the following way.
As $f \in S\lbrace m_1,0\rbrace$ property (4) implies  for $i=1,\dots,\ell$ that
\begin{equation}
\vert \partial_{t}^{j_{i}}\partial_{\xi}^{\alpha_{i}}f\vert \leq C_{j_{i},\alpha_{i}}\big(\Phi_{\epsilon }(t-1)+1\big)^{j_{i}}|\xi|^{m_1+\vert\alpha_{i}\vert}.
\end{equation}
Therefore,
\begin{equation}
\vert \partial_{t}^{j_{1}}\partial_{\xi}^{\alpha_{1}}f\cdots \partial_{t}^{j_{\ell}}\partial_{\xi}^{\alpha_{\ell}}f\vert \lesssim 
\big(\Phi_{\epsilon }(t-1)+1\big)^{j_{1}+\cdots+j_{\ell}}|\xi|^{\ell m_{1}+\vert\alpha_{1}\vert+\dots+\vert\alpha_{\ell}\vert}
\end{equation}
and using the condition $\vert f(t,\xi,\epsilon)\vert\textgreater c|\xi|^{m_1}$ we obtain
\begin{equation}
\begin{split}
\bigg|\frac{\partial_{t}^{j_{1}}\partial_{\xi}^{\alpha_{1}}f\cdots\partial_{t}^{j_{\ell}}\partial_{\xi}^{\alpha_{\ell}}f}{f^{\ell+1}}\bigg| &\lesssim \frac{\big(\Phi_{\epsilon }(t-1)+1\big)^{j_{1}+\cdots+j_{\ell}}|\xi|^{\ell m_1+\vert\alpha_{1}\vert+\cdots+\vert\alpha_{\ell}\vert}}{|\xi|^{m_{1}(\ell+1)}}
\lesssim \frac{\big(\Phi_{\epsilon }(t-1)+1\big)^{k}|\xi|^{\ell m_1+\vert\alpha\vert}}{|\xi|^{m_{1}(\ell+1)}}\\
&\lesssim \big(\Phi_\epsilon(t-1)+1\big)^{k}|\xi|^{-m_1+\vert\alpha\vert}.
\end{split}
\end{equation}
Summing all these terms yields the desired estimate. The boundary estimate follows similarly.
\end{proof}

These symbol classes and in particular the embeddings
\begin{equation}\label{eq:hyp-emb}
  \mathcal S\{-1,2\} \hookrightarrow \mathcal S\{0,1\} \hookrightarrow \mathcal S\{1,0\}
\end{equation}
will be of importance for the treatment within the hyperbolic zone. The gain of decay in $|\xi|$ will be paid for by a loss of point-wise control in the $t$-variable near the singularity. What we gain, are integrability properties and improved limits at the regular face.

\begin{prop}\label{prop:Symb. from}
Within the hyperbolic zone $Z_{\rm hyp}(N)$,
\begin{enumerate}
\item[\bf(1)] symbols from $\mathcal S\left\lbrace0,0\right\rbrace$ are uniformly bounded;
\item[\bf(2)] symbols from $\mathcal S\left\lbrace0,1\right\rbrace$  are uniformly integrable with respect to $t$;
\item[\bf(3)] symbols $a\in \mathcal S\left\lbrace-1,2\right\rbrace$ satisfy
\begin{equation}
    \int_0^t |a(\theta,\xi,\epsilon)| \d \theta \le C |\xi|^{-1}\big(\Phi_{\epsilon }(t-1)+1\big)
\end{equation}
for all $0<t\le t_{\xi_1}$, and 
\begin{equation}
    \int_t^2 |a(\theta,\xi,\epsilon)| \d \theta \le C |\xi|^{-1}\big(\Phi_{\epsilon }(t-1)+1\big)
\end{equation}
for all $t_{\xi_2} \le t\le 2$.
\end{enumerate}
\end{prop}

\begin{proof}
Statement (1) is obvious from the definition of the symbol class. Next we prove (2). If $f \in \mathcal S\left\lbrace0,1\right\rbrace$ then it satisfies the point-wise estimate
\begin{equation}
\vert f(t,\xi,\epsilon)\vert \leq C\big(\Phi_{\epsilon }(t-1)+1\big),
\end{equation}
and therefore after integrating over $t\in[0,t_{\xi_1}(\epsilon)]$  (or similarly over $t\in [t_{\xi_2}(\epsilon),2]$) one has
\begin{equation}
\begin{split}
\int_0^{t_{\xi_1}}\vert f(t,\xi,\epsilon)\d s\vert \d t& \leq  C \int_0^{t_{\xi_1}} \Phi_{\epsilon}(t-1)\d t +C  \int_0^{t_{\xi_1}} \d t = C\int_{-\epsilon^{-1}}^{\epsilon^{-1}(t_{\xi_1}-1)}\Phi(\tau)\d \tau + C \le C \bigg[1+\int_{-\infty}^0 \Phi(\tau)\d \bigg]
\end{split}
\end{equation}
for any fixed $\epsilon\in(0,1]$ and $\xi\in\R^n$. It remains to prove (3). If $a\in S\left\lbrace-1,2\right\rbrace$ then it satisfies the point-wise estimate
\begin{equation}
\vert a(t,\xi,\epsilon)\vert \leq C|\xi|^{-1} \big(\Phi_\epsilon(t-1)+1\big)^2
\end{equation}
and the only new term needing to be treated is the one arising from the square of the shape function. This can be estimated by means of 
\eqref{eq:PhiSquEst1} for $t<1$ as
\begin{equation}
\begin{split}
 \int_0^t \Phi_\epsilon^2(\theta) \d\theta &= \epsilon^{-1}  \int_{-\epsilon^{-1}}^{\epsilon^{-1} (t-1) }  \Phi^2(\tau) \d\tau 
 \le C  \epsilon^{-1}  \int_{-\epsilon^{-1}}^{\epsilon^{-1} (t-1) }  \Phi'(\tau) \d\tau 
 \le C \epsilon^{-1} \Phi(\epsilon^{-1} (t-1) ) - C \epsilon^{-1} \Phi(-\epsilon^{-1}) \le C \Phi_\epsilon(t-1)  
\end{split}
\end{equation}
and similarly for the case $t>1$.
\end{proof}

\subsubsection{Transformations}
Within the hyperbolic zone, we apply transformations to our system in order to extract precise information about the behaviour of its fundamental solution. Recall that \eqref{Transf to system} is of the form
$\mathrm D_t U = (A + B) U$ with $A\in\mathcal S\{1,0\}$ and $B\in\mathcal S\{0,1\}$.
Using the diagonaliser of the principal part $A$
\begin{equation}\label{eq:Mdef}
M=\frac{1}{\sqrt{2}}\begin{pmatrix}
1 & -1 \\ 
1 & 1
\end{pmatrix}\qquad \text{  with  inverse  }\qquad
M^{-1}=\frac{1}{\sqrt{2}}\begin{pmatrix}
1 & 1 \\ 
-1 & 1
\end{pmatrix},
\end{equation}
the matrix $A$ can be written as
\begin{equation}
A(\xi)=M  D(\xi) M^{-1}
\end{equation}
with $D(\xi) = \diag(|\xi|,-|\xi|)$. Hence, setting $V(t,\xi,\epsilon)=M^{-1}U(t,\xi,\epsilon)$ system \eqref{Transf to system} can be rewritten as
\begin{equation}
\D_{t}V(t,\xi,\epsilon )=\big( D(\xi)+R(t,\epsilon) \big) V(t,\xi,\epsilon ) \label{Syst after 1 diag}
\end{equation}
with a remainder given by
\begin{equation}
R(t,\epsilon)=M^{-1}B(t,\epsilon)M=\frac{\i}{2}\mathfrak d_{\epsilon}(t)\begin{pmatrix} 
1 & 1 \\ 
1 & 1
\end{pmatrix}
\in\mathcal S\{0,1\}.
\end{equation}
Our aim is to further improve the remainder within the hyperbolic hierarchy \eqref{eq:hyp-emb}. This allows to extract more detailed information on the propagation of singularities close to the singularity later. For this we follow \cite{RW14} to construct 
 transformation matrices $N_{k}(t,\xi,\epsilon)$, transforming the system (\ref{Syst after 1 diag}) into a new system with an updated diagonal part and an improved remainder. The construction is done in such a way, that the operator identity
\begin{equation}
\big(\D_{t}-D(\xi )-R(t,\epsilon )\big)N_{k}(t,\xi,\epsilon)=N_{k}(t,\xi,\epsilon)\left(D_{t}-D_k(t,\xi,\epsilon)-R_{k}(t,\xi,\epsilon)\right) \label{Identity}
\end{equation}
holds for $k\ge1$ and
\begin{enumerate}
\item[\bf(1)] the matrix-valued symbols $D_k(t,\xi,\epsilon)$ are given by
\begin{equation} D_k(t,\xi,\epsilon) = D(\xi) + F^{(0)}(t,\xi,\epsilon)+\cdots +F^{(k-1)}(t,\xi,\epsilon) \end{equation}
with diagonal $F^{(j)}(t,\xi,\epsilon) \in S\{-j,j+1\}$;
\item[\bf(2)] the transformation matrices $N_k(t,\xi,\epsilon)$ are of the form
\begin{equation}
   N_k(t,\xi,\epsilon)  = \mathrm I + N^{(1)}(t,\xi,\epsilon) + \cdots + N^{(k)}(t,\xi,\epsilon) 
\end{equation}
with $N^{(j)}(t,\xi,\epsilon) \in S\{-j,j\}$;
\item[\bf(3)] the remainder satisfies $R_k(t,\xi,\epsilon) \in S\{-k,k+1\}$. 
\end{enumerate}
We give the construction for $k=1$ in full detail. In this case \eqref{Identity} simplifies modulo $S\{-1,2\}$ to 
the commutator equation
\begin{equation}
\big[D(\xi),N^{(1)}(t,\xi,\epsilon)\big]
=F^{(0)}(t,\xi,\epsilon)-R(t,\epsilon).  \label{Commutator}
\end{equation}
As the diagonal part of the commutator vanishes, we set 
\begin{equation}\label{eq:4.34}
F^{(0)}(t,\epsilon)=\diag R(t,\epsilon)=\frac{\i}{2}\mathfrak d_{\epsilon}(t)
\begin{pmatrix} 1 & 0\\ 
0 & 1
\end{pmatrix} \in S\{0,1\},
\end{equation}
and determine the off-diagonal entries of 
\begin{equation}
N^{(1)}(t,\xi,\epsilon) = \begin{pmatrix}
n_{11} & n_{12}\\ 
n_{21} & n_{22}
\end{pmatrix}
\end{equation}
from (\ref{Commutator}) as $n_{12}=-\frac{\i}{4|\xi|}\mathfrak d_{\epsilon}(t)$ and $n_{21}=\frac{\i}{4|\xi|}\mathfrak d_{\epsilon}(t)$. The diagonal entries are chosen to be zero, and hence
\begin{equation}\label{eq:N1}
N^{(1)}(t,\xi,\epsilon)=\frac{\i}{4|\xi|}\mathfrak d_{\epsilon}(t)\begin{pmatrix}
0 & -1\\
1 & 0
\end{pmatrix}\in S\left\{-1,1\right\}.
\end{equation}
The transformation matrix $N_{1}(t,\xi,\epsilon)=\mathrm I+N^{(1)}(t,\xi,\epsilon)$ is invertible, provided that the zone constant is chosen large enough.

\begin{prop}\label{prop:Diagonalisation}
Assume Hypotheses {\bf (H1)} and {\bf (H2)} hold. Then, there exists a matrix $N^{(1)}(t,\xi,\epsilon)\in S\left\{-1,1\right\}$ and a diagonal matrix $F^{(0)}(t,\epsilon )\in S\left\{0,1\right\}$, such that the identity (\ref{Identity}) is satisfied with a remainder $R_{1}(t,\xi,\epsilon)\in
S\left\{-1,2\right\}$. Moreover, we can find a zone constant $N$, such that the transformation matrix $N_{1}(t,\xi,\epsilon)=I+N^{(1)}(t,\xi,\epsilon)$ is invertible in $Z_{\rm hyp}(N)$ and $N_{1}^{-1}(t,\xi,\epsilon)\in S\left\{0,0\right\}$.
\end{prop}

\begin{proof}
It remains to show the invertibility of $N_1(t,\xi,\epsilon)$. Indeed, by \eqref{eq:N1} it follows that
\begin{equation}
\det N_{1}=1-\frac{\mathfrak d_{\epsilon}^{2}(t)}{16|\xi|^2} 
\end{equation}
and by Lemma~\ref{lem:B_epsilon} one has
\begin{equation}
\frac{\mathfrak d_{\epsilon}^{2}(t)}{16|\xi|^2} \leq \frac{\big(c_{1}\Phi_{\epsilon}(t-1)+c_{2}\big)^{2}}{16|\xi|^2}.
\end{equation}
Hence, by choosing  the zone constant $N$ large enough such that
\begin{equation}
c_{1}\psi_{\epsilon}(t-1)+c_{2} \leq N\psi_{\epsilon}(t-1)+N,
\end{equation}
the invertibility follows. By the calculus rules of Proposition~\ref{prop:Symbolic properties}, we also conclude $N_1^{-1}\in S\{0,0\}$. 

The matrices $N^{(1)}(t,\xi,\epsilon) \in S\{-1,1\}$ and $F^{(0)}(t,\epsilon) \in S\{0,1\}$ are already constructed in such a way that \eqref{Identity} holds with the remainder
\begin{equation}\label{eq:4.40}
\begin{split}
   R_1(t,\xi,\epsilon)  = N_1^{-1}(t,\xi,\epsilon)\big( 
      R(t,\epsilon)N^{(1)}(t,\xi,\epsilon)   - \D_tN^{(1)} (t,\xi,\epsilon) - N^{(1)}(t,\xi,\epsilon) F^{(0)}(t,\epsilon) 
   \big) \in S\{-1,2\}
\end{split}
\end{equation}
and the statement is proved.
\end{proof}

\begin{rem}
Taking the limits $\epsilon\to0$ at the regular faces of $Z_{\rm hyp}(N)$ the diagonalisation procedure yields in particular the transformations needed to construct representations of solutions in the case of smooth coefficients. In particular 
the limit
\begin{equation}
N^{(1)}(t,\xi,0) = \lim_{\epsilon\to0} N^{(1)}(t,\xi,\epsilon)
\end{equation} 
exists for $t\ne1$ and $|\xi|>N$ and satisfies
\begin{equation}\label{eq:N1epsEst}
   \| N_1(t,\xi,\epsilon) - N_1(t,\xi,0) \| 
   = 
   \| N^{(1)}(t,\xi,\epsilon) - N^{(1)}(t,\xi,0) \| \le C \epsilon |\xi|^{-1},
\end{equation}
and as the inverse $N_1^{-1}(t,\xi,\epsilon)$ can be written as a Neumann series, we know that $N_1^{-1}(t,\xi,\epsilon)- \mathrm I \in \mathcal S\{-1,1\}$ and consequently 
\begin{equation}\label{eq:N1epsInvEst}
   \| N_1^{-1}(t,\xi,\epsilon)- N_1(t,\xi,0)^{-1} \| 
    \le C \epsilon |\xi|^{-1}.
\end{equation}
Similarly the limit $R_1(t,\xi,0) = \lim_{\epsilon\to0} R_1(t,\xi,\epsilon)$ satisfies
\begin{equation}\label{eq:R1epsEst}
   \| R_1(t,\xi,\epsilon) - R_1(t,\xi,0)\| \le C \epsilon |\xi|^{-1}. 
\end{equation}
This enables us to relate the construction of fundamental solutions for the regularised family to the fundamental solution of the original problem outside the singularity.
\end{rem}

\subsubsection{Fundamental solution to the diagonalized system} We now fix the zone constant $N$ large enough to guarantee that $N_{1}(t,\xi,\epsilon)$ be uniformly invertible within the hyperbolic zone $Z_{\rm hyp}(N)$. Then for $V$ solving \eqref{Syst after 1 diag}, the transformed function
\begin{equation}
V_{1}(t,\xi,\epsilon )=N_{1}^{-1}(t,\xi,\epsilon)V(t,\xi,\epsilon)
\end{equation}
satisfies due to \eqref{Identity}
\begin{equation}
\D_{t}V_{1}(t,\xi)=\left(D(\xi)+F^{(0)}(t,\xi,\epsilon)+R_{1}(t,\xi,\epsilon)\right)V_{1}(t,\xi) \label{Syst after 2 diag}
\end{equation}
with the diagonal matrix $F^{(0)}\in S\left\{0,1\right\}$ given by \eqref{eq:4.34} and the remainder $R_{1}(t,\xi,\epsilon)\in S\left\{-1,2\right\}$ specified by \eqref{eq:4.40}. We construct its fundamental solution.

\begin{thm}\label{thm:Fundamental sol in Zhyp}
Assume the Hypotheses \textbf{(H1)} and \textbf{(H2)} hold. Then the fundamental solution $
\mathcal E_{1}(t,s,\xi,\epsilon)$ to the transformed system (\ref{Syst after 2 diag}) can be represented by
\begin{equation}
\mathcal E_{1}(t,s,\xi,\epsilon)=\sqrt{\frac{b_{\epsilon}(s)}{b_{\epsilon}(t)}}\mathcal E_{0}(t,s,\xi)\mathcal Q(t,s,\xi,\epsilon)
\end{equation}
for $[s,t] \times \{(\xi,\epsilon)\}\subset Z_{\rm hyp}(N)$, where
\begin{enumerate}
\item[\bf(1)] the factor $\sqrt{\frac{b_{\epsilon}(s)}{b_{\epsilon}(t)}}$ describes the main influence of the dissipation term;
\item[\bf(2)] the matrix $\mathcal E_{0}(t,s,\xi)$ is the fundamental solution of the hyperbolic principal part $\D_{t}-D(\xi)$ given by
\begin{equation}\label{eq:E0}
\mathcal E_{0}(t,s,\xi)=\begin{pmatrix}
\exp(\i(t-s)|\xi|) & 0\\
0 & \exp(-\i(t-s)|\xi|)
\end{pmatrix};
\end{equation}
\item[\bf(3)] the matrix $\mathcal Q(t,s,\xi,\epsilon)$ 
is uniformly bounded
\begin{equation}
   \| \mathcal Q(t,s,\xi,\epsilon) \| \le \exp\bigg( \int_s^t \| R_1(\tau,\xi,\epsilon) \| \d\tau \bigg),
\end{equation}
uniformly invertible within the hyperbolic zone due to
\begin{equation}
   |\det \mathcal Q(t,s,\xi,\epsilon) | \ge  \exp\bigg( \int_s^t \|R_1(\tau,\xi,\epsilon)\|\d\tau\bigg),
\end{equation}
and has the precise behaviour for large $|\xi|$ determined by the identity matrix
\begin{equation}
   \| \mathcal Q(t,s,\xi,\epsilon) - \mathrm I \| \le \int_s^t \|R_1(t_1,\xi,\epsilon)\|  \exp\bigg( \int_s^{t_1} \| R_1(\tau,\xi,\epsilon) \| \d\tau \bigg) \d t_1  .
\end{equation}
\end{enumerate}
\end{thm}

\begin{proof}
We consider first $\D_{t}-D(\xi)-F^{0}$, i.e. the main diagonal part of the transformed system (\ref{Syst after 2 diag}). Its fundamental solution is given by
\begin{equation}
\mathcal E_{0}(t,s,\xi) \exp\bigg(-\frac12 \int_s^t \mathfrak d_\epsilon(\tau)\d\tau\bigg)  =  
\sqrt{\frac{b_{\epsilon}(s)}{b_{\epsilon}(t)}} \mathcal E_0(t,s,\xi,\epsilon),
\end{equation}
where $\mathcal E_{0}(t,s,\xi)$ is the fundamental solution to $\D_{t}-D(\xi)$ given by \eqref{eq:E0}. For the fundamental solution to the system (\ref{Syst after 2 diag}) we use an ansatz in the form
\begin{equation}
\mathcal E_1(t,s,\xi) =  \sqrt{\frac{b_{\epsilon}(s)}{b_{\epsilon}(t)}}  \mathcal E_0(t,s,\xi,\epsilon) \mathcal Q(t,s,\xi,\epsilon)
\end{equation}
for a still to be determined matrix $\mathcal Q(t,s,\xi,\epsilon)$. A simple calculation shows that $\mathcal Q(t,s,\xi,\epsilon)$ must solve
\begin{equation}
\D_{t}\mathcal Q(t,s,\xi,\epsilon)=\mathcal{R}(t,s,\xi,\epsilon)\mathcal Q(t,s,\xi,\epsilon),\qquad  \mathcal Q(s,s,\xi,\epsilon)=\mathrm I,
\end{equation}
with coefficient matrix
\begin{equation}
\mathcal{R}(t,s,\xi,\epsilon)=\mathcal E_{0}(s,t,\xi)R_{1}(t,\xi,\epsilon)\mathcal E_{0}(t,s,\xi)
\end{equation}
determined by the remainder $R_1$ and the fundamental solution of the hyperbolic principal part $\mathcal E_0$. The solution $\mathcal Q$ can thus be represented in terms of the Peano--Baker series 
\begin{equation}
\mathcal Q(t,s,\xi,\epsilon)=\mathrm I+\sum_{k=1}^{\infty}\i^{k}\int_{s}^{t}\mathcal{R}(t_{1},s,\xi,\epsilon)\int_{s}^{t_1}\mathcal{R}(t_{2},s,\xi,\epsilon)\int_{s}^{t_2}
	\dots \int_{s}^{t_{k-1}}\mathcal{R}(t_{k},s,\xi,\epsilon) \d t_{k} \dots \d t_{2} \d t_{1}, \label{Formula for Q(epsilon)}
\end{equation}
and it remains to provide estimates based on this series representation. As $\mathcal E_0$ is unitary, we obtain from the symbol estimate of the remainder $R_1$ 
\begin{equation}
\begin{split} 
 \| \mathcal R(t,s,\xi,\epsilon) \| &= \| R_1(t,\xi,\epsilon) \|  \le C |\xi|^{-1} \left( \Phi_\epsilon(t-1) + 1 \right)^2 
 \le \frac CN \left( \Phi_\epsilon(t-1) + 1 \right),
 \end{split}
\end{equation}
and thus it follows that
\begin{equation}
 \| \mathcal Q(t,s,\xi,\epsilon) \| \le \exp\bigg( \int_s^t \| R_1(\tau,\xi,\epsilon) \| \d\tau \bigg) \le \exp(C/N).
\end{equation}
Together with
\begin{equation}
   \det \mathcal Q(t,s,\xi,\epsilon) = \exp \bigg( \int_s^t \mathrm{trace}\, R_1(\tau,\xi,\epsilon) \d\tau \bigg)
\end{equation}
the uniform invertibility of $\mathcal Q$ follows. Furthermore, by using \eqref{eq:PhiSquEst1} we obtain
\begin{equation}\label{eq:Q-I_Est}
  \| \mathcal Q(t,s,\xi,\epsilon)-\mathrm I \| \le 
  C |\xi|^{-1}  \big(\Phi_\epsilon(t-1)+1\big) \exp (C/N),
\end{equation}
and the main contribution of $\mathcal Q$ for large $|\xi|$ is given by the identity matrix.
 \end{proof}
 
\subsubsection{Fundamental solution to the original system} 
After obtaining the fundamental solution to the transformed system (\ref{Syst after 2 diag}), we go back to the original problem (\ref{Transf to system}) and obtain in the hyperbolic zone the representation
\begin{equation}\label{eq:4.64}
\mathcal E_{\rm hyp}(t,s,\xi,\epsilon)=\sqrt{\frac{b_{\epsilon}(s)}{b_{\epsilon}(t)}}MN_{1}(t,\xi,\epsilon)\mathcal E_{0}(t,s,\xi) \mathcal Q(t,s,\xi,\epsilon)N_{1}^{-1}(s,\xi,\epsilon)M^{-1}
\end{equation}
for the fundamental solution. We will briefly discuss its limiting behaviour as $\epsilon\to 0$ for fixed $s<t<1$ or $1<s<t$. As $\mathcal E_0(t,s,\xi)$ is independent of $\epsilon$ and the transformation matrix $N_1(t,\xi,\epsilon)$ is already estimated by
\eqref{eq:N1epsEst}, this boils down to considering $\mathcal Q(t,s,\xi,\epsilon)$.

\begin{lem}\label{lem: Estimate N1, N1 inverse and Q}
The limit 
\begin{equation}
\mathcal Q(t,s,\xi,0) = \lim_{\epsilon\to0} \mathcal Q(t,s,\xi,\epsilon)
\end{equation}
exists for fixed $s<t<1$ or $1<s<t$, is uniformly bounded and invertible, and satisfies the estimate
\begin{equation}\label{Estimate Q}
\Vert \mathcal Q(t,s,\xi,\epsilon)-\mathcal Q(t,s,\xi,0)\Vert\lesssim \epsilon|\xi|^{-1}
\end{equation}
holds for all $[s,t]\times\{(\xi,\epsilon)\}\subset Z_{hyp}$ with the condition that $\min \{ \vert t-1\vert, |s-1| \} \ge \epsilon K$.
\end{lem}
\begin{proof}
We use \eqref{eq:R1epsEst} in combination with \eqref{Formula for Q(epsilon)} and consider
\begin{equation}
\mathcal Q(t,s,\xi,0)= \mathrm I+\sum_{k=1}^{\infty}\i^{k}\int_{s}^{t}\mathcal{R}(t_{1},s,\xi,0)\int_{s}^{t_1}\mathcal{R}(t_{2},s,\xi,0)\int_{s}^{t_2}   \dots \int_{s}^{t_{k-1}}\mathcal{R}(t_{k},s,\xi,0)\d t_{k}\dots \d t_{2}\d t_{1} \label{Formula for Q(0)}
\end{equation}
defined in terms of
\begin{equation}
\mathcal{R}(t,s,\xi,0)=\mathcal E_{0}(s,t,\xi)R_{1}(t,\xi,0)\mathcal E_{0}(t,s,\xi)
\end{equation}
with
\begin{equation}
   \|\mathcal R(t,s,\xi,0)\|=\|R_1(t,\xi,0)\| \le C |\xi|^{-1}
\end{equation}
uniformly in $0<s<t<1$ or $1<s<t<2$ and $|\xi|>N$. It thus follows that $\mathcal Q(t,s,\xi,0)$ is uniformly bounded and uniformly invertible. To estimate the difference between $\mathcal Q(t,s,\xi,\epsilon)$ and $\mathcal Q(t,s,\xi,0)$, we use a perturbation argument based on the estimate
\begin{equation}
   \|\mathcal R(t,s,\xi,\epsilon) - \mathcal R(t,s,\xi,0)\| =
   \|R_1(t,\xi,\epsilon) - R_1(t,\xi,0)\| \le C \epsilon |\xi|^{-1}
\end{equation}
for $|t-1| \ge \epsilon K$ following from \eqref{eq:R1epsEst}. Differentiating 
\begin{equation}
 \mathcal Q(t,s,\xi,\epsilon) = \mathcal Q(t,s,\xi,0) \Xi(t,s,\xi,\epsilon) 
\end{equation}
yields for $\Xi(t,s,\xi,\epsilon)$ the equation
\begin{equation}
 \D_t  \Xi(t,s,\xi,\epsilon) = \mathcal Q(t,s,\xi,0) \big(\mathcal R(t,s,\xi,\epsilon) - \mathcal R(t,s,\xi,0) \big) \mathcal Q(s,t,\xi,0) 
 \Xi(t,s,\xi,\epsilon)
\end{equation}
with initial condition $\Xi(s,s,\xi,\epsilon)=\mathrm I$ and coefficient matrix estimated by 
\begin{equation}
\| \mathcal Q(t,s,\xi,0) \big(\mathcal R(t,s,\xi,\epsilon) - \mathcal R(t,s,\xi,0) \big) \mathcal Q(s,t,\xi,0) \| \le C  \epsilon |\xi|^{-1}.
\end{equation}
Therefore, using the representation of $\Xi(t,s,\xi,\epsilon)$ in terms of  the Peano--Baker series we obtain the estimate
\begin{equation}
 \|\Xi(t,s,\xi,\epsilon)\| \le \exp\left(C\epsilon |\xi|^{-1} |t-s| \right) = 1 + \mathcal O(\epsilon |\xi|^{-1})
\end{equation}
uniform with respect to $t$ and $s$ for $|t-1|,|s-1|\ge \epsilon K$ and thus the desired statement follows.
\end{proof}

\begin{prop}\label{prop:Estimate E_{hyp}}
The estimate
\begin{align}\label{eq:4.76}
\Vert \mathcal E_{\rm hyp}(t,s,\xi,\epsilon)-\mathcal E_{\rm hyp}(t,s,\xi,0)\Vert &\lesssim 
\epsilon
\end{align}
holds for all $[s,t]\times\{(\xi,\epsilon)\}\subset Z_{hyp}$ with the condition that $\min \{ \vert t-1\vert, |s-1| \} \ge \epsilon K$.\end{prop}

\begin{proof}
The proof follows directly from the representation \eqref{eq:4.64} of $\mathcal E_{\rm hyp}(t,s,\xi,\epsilon)$ combined with an 
analogous formula for the limit $\mathcal E_{\rm hyp}(t,s,\xi,0) = \lim_{\epsilon\to0} \mathcal E_{\rm hyp}(t,s,\xi,\epsilon)$. As all terms in \eqref{eq:4.64} are uniformly bounded within the hyperbolic zone we obtain
\begin{equation}
\begin{split}
 \| \mathcal E_{\rm hyp}(t,s,\xi,\epsilon)-\mathcal E_{\rm hyp}(t,s,\xi,0)\|
 \lesssim& \left| \sqrt{\frac{b_{\epsilon}(s)}{b_{\epsilon}(t)}} - \sqrt{\frac{b_{\epsilon}(0)}{b_{\epsilon}(0)}} \right|
 + \| N_1(t,\xi,\epsilon) - N_1 (t,\xi,0) \| 
 \\&+ \|\mathcal Q(t,s,\xi,\epsilon) - \mathcal Q(t,s,\xi,0) \|
 + \| N_1^{-1}(s,\xi,\epsilon) - N_1^{-1} (s,\xi,0) \| 
 \end{split}
\end{equation}
Each of the last three differences appearing on the right hand side can be controlled by $\epsilon |\xi|^{-1}$ by
estimate \eqref{eq:N1epsEst} for $N_1(t,\xi,\epsilon)$ and \eqref{eq:N1epsInvEst} for $N_1^{-1}(s,\xi,\epsilon)$, and by estimate
 \eqref{Estimate Q} for $\mathcal Q(t,s,\xi,\epsilon)$. Furthermore, by Proposition~\ref{lem:b_epsilon-b_0}
 for $b_\epsilon(s)$ and $b_\epsilon(t)$, we know that the first difference is controlled by $\epsilon$. The desired estimate for the fundamental solution follows.
\end{proof}

\subsection{Treatment in the singular zone}
Now we consider equation \eqref{Eqn in phase space} within the singular zone.  In order to describe its fundamental solution we use the substitution $\tau=\epsilon^{-1}(t-1)$ and replace the parameter $|\xi|$ by $\Lambda=\epsilon|\xi|$. Then the equation \eqref{Eqn in phase space} can be rewritten as
\begin{equation}
\hat{u}_{\tau\tau}+\Lambda^{2}\hat{u}+\beta_{\epsilon}(\tau)\hat{u}_{\tau}=0, \label{Eqn in Zsing}
\end{equation}
where 
\begin{equation} \label{Beta_epsilon}
\beta_{\epsilon}(\tau) 
=\epsilon \mathfrak d_\epsilon(1+\epsilon \tau) 
=\frac{\int_{-\infty}^\infty b(1+\epsilon(\tau-\theta))\psi^{\prime}(\theta)\d\theta}{\int_{-\infty}^{\infty}b(1+\epsilon(\tau-\theta))\psi(\theta)\d\theta}.
\end{equation}
We recall here that  in the new coordinates the singular zone is rewritten as
\begin{equation}
Z_{\rm sing}(N)=\left\{( \tau,\Lambda,\epsilon) \mid \Lambda \leq N \Phi(\tau)+N\epsilon\right\rbrace
\end{equation}
so that the interval in which to solve our equation is given by $[\tau_{\Lambda_{1}}(\epsilon),\tau_{\Lambda_{2}}(\epsilon)]$ with implicitly defined endpoints through 
\begin{equation}
\Lambda = N\Phi(\tau_\Lambda)+\epsilon.
\end{equation}

\subsubsection{System form}
Reformulating our equation as a system in 
\begin{equation}
U(\tau,\Lambda,\epsilon)=\begin{pmatrix} \Lambda\hat{u} \\ \partial_{\tau}\hat{u} \end{pmatrix}
\end{equation}
yields
\begin{equation}
\partial_{\tau}U(\tau,\Lambda,\epsilon)=\left[ \begin{pmatrix}
0 & 0 \\ 
0 & -\beta_{\epsilon}(\tau)
\end{pmatrix} + \begin{pmatrix}
0 & \Lambda \\ 
-\Lambda & 0
\end{pmatrix}\right]U(\tau,\Lambda,\epsilon),
 \label{Transf to system in Zsing}
\end{equation}
where now the first matrix $A(\tau,\epsilon) = \diag(0,-\beta_\epsilon(\tau))$ 
is treated as the dominant part and the second matrix 
\begin{equation}
 \begin{pmatrix}
0 & \Lambda \\ 
-\Lambda & 0
\end{pmatrix} = \Lambda \begin{pmatrix}0&1\\-1&0\end{pmatrix}
 = \Lambda \mathrm J
 \end{equation}
plays the role of the remainder. 

\begin{prop}\label{prop:Limiting behaviour of Beta}
Under the assumptions {\bf (H1)} and {\bf (H2)} for the function $b$
\begin{equation} \label{Limiting behaviour of Beta}
\beta_{\epsilon}(\tau)=\beta_{0}(\tau)+\mathcal{O}(\epsilon)
\end{equation}
holds uniformly with respect to $\tau$, where   $\beta_{0}(\tau)$ is given by
\begin{equation}\label{eq:beta0eq}
\beta_{0}(\tau)=\frac{h\psi(\tau)}{h\int_{-K}^{\tau}\psi(\theta)\d \theta+b(1_{-0})}
=\frac{h\psi(\tau)}{b(1_{+0}) -h\int_{\tau}^{K}\psi(\theta)\d \theta}
\end{equation}
in terms of $h=b(1_{+0}) - b(1_{-0})$  the jump of $b$ at $t=1$. 
\end{prop}

\begin{rem}
In  \eqref{eq:beta0eq} the numerator is the derivative of the denominator with respect to $\tau$. We also see that $\beta_{0}(\tau)$ is compactly supported with $\supp\beta_{0}=\supp\psi = [-K,K]$.
\end{rem}

\begin{proof}
The statement follows by considering both the numerator and the denominator of the representation
\eqref{Beta_epsilon} separately. First,
\begin{equation} \label{Estimate for Beta's numerator}
\begin{split}
\int_{-\infty}^\infty b(1+\epsilon(\tau-\theta))\psi^{\prime}(\theta)\d\theta
&=\int_{-\infty}^{\tau}b(1+\epsilon(\tau-\theta))\psi^{\prime}(\theta)\d\theta  + \int_{\tau}^{+\infty}b(1+\epsilon(\tau-\theta))\psi^{\prime}(\theta)\d\theta\\
&= h\psi(\tau)+\epsilon \int_{-\infty}^{\tau}  b^{\prime}(1+\epsilon(\tau-\theta))\psi(\theta)\d\theta
 + \epsilon \int_{\tau}^{+\infty} b^{\prime}(1+\epsilon(\tau-\theta))\psi(\theta)\d\theta \\
&=h\psi(\tau)+\mathcal O(\epsilon)
\end{split}
\end{equation}
using integration by parts and the fact that $b'$ is bounded on both $[0,1]$ and $[1,2]$. Similarly, we obtain for the denominator
\begin{equation}
\begin{split}
 \bigg|\int_{-\infty}^\infty &b(1+\epsilon(\tau-\theta))\psi(\theta)\d\theta - \int_{-\infty}^{\tau}b(1_{+0})\psi(\theta)\d\theta - \int_{\tau}^{+\infty}b(1_{-0})\psi(\theta)\d\theta\bigg| \\
&  \leq \int_{-\infty}^{\tau}\big| b(1+\epsilon(\tau-\theta))-b(1_{+0})\big|\psi(\theta)\d\theta 
+ \int_{\tau}^{\infty}\big| b(1+\epsilon(\tau-\theta))-b(1_{-0}))\big|\psi(\theta)\d\theta \\
&  \leq \int_{-K}^{\tau}C_{1}\epsilon\vert\tau-\theta\vert\psi(\theta)\d\theta + \int_{\tau}^{K}C_{2}\epsilon\vert\tau-\theta\vert\psi(\theta)\d\theta ,
\end{split}
\end{equation}
where for the last line we applied the mean value theorem to the function $b$ fir
$C_{1}=\sup_{s\in[1,2]}\vert b^{\prime}(s)\vert$ and $C_{2}=\sup_{s\in[0,1]}\vert b^{\prime}(s)\vert$.
Hence
\begin{equation}\label{Estimate for Beta's denominator}
\int_{-\infty}^\infty b(1+\epsilon(\tau-\theta))\psi(\theta)\d\theta
= b(1_{+0}) \int_{-\infty}^{\tau}\psi(\theta)\d\theta  + b(1_{-0}) \int_{\tau}^{+\infty}\psi(\theta)\d\theta
+\mathcal O(\epsilon),
\end{equation}
and therefore by combining \eqref{Estimate for Beta's numerator} and \eqref{Estimate for Beta's denominator} the desired statement follows.
\end{proof}

\subsubsection{Construction of the fundamental solution in the singular zone}
In the following we want to derive properties of  the fundamental solution to \eqref{Transf to system in Zsing}. The strategy is again to use a perturbation argument to incorporate the remainder terms. Note, that in singular variables both $\tau$ and $\Lambda$ stay bounded within $Z_{\rm sing}(N)$ and our main interest is in the characterisation of the solution when $\Lambda\to0$ and $\epsilon\to0$.

\begin{thm}\label{thm:Fundamental sol in Zsing}
The fundamental solution to the system (\ref{Transf to system in Zsing}) can be represented by
\begin{equation}\label{eq:Esing-rep}
\mathcal E_{\rm sing}(\tau,\theta,\Lambda,\epsilon)=\mathcal F(\tau,\theta,\epsilon)\mathcal G(\tau,\theta,\Lambda,\epsilon)
\end{equation}
for $[\theta,\tau]\times \{(\Lambda,\epsilon)\}\subset Z_{\rm sing}(N)$ with $\theta<\tau$, where
\begin{enumerate}
\item[\bf(1)] $\mathcal F(\tau,\theta,\epsilon)$ is the fundamental solution to the main part $\partial_{\tau}-\diag(0,-\beta_\epsilon(\tau))$
 given by
\begin{equation}
\mathcal F(\tau,\theta,\epsilon)=\begin{pmatrix}
1 & 0 \\ 
0 &   \exp\left(-\int_{\theta}^{\tau} \beta_\epsilon(\vartheta)\d\vartheta \right)
\end{pmatrix}
\end{equation}
with
\begin{equation}\label{eq:4.90}
  \exp\bigg(\int_{\theta}^{\tau} \beta_\epsilon(\theta)\d\theta \bigg)
  =
 \frac{h \Theta(\tau) + b(1_{-0})}{ h \Theta(\theta) + b(1_{-0})}
  \left(1+ \mathcal O(\epsilon)\right)
\end{equation}
 in terms of the smoothed Heaviside function
\begin{equation} \Theta(\tau) =  \int_{-\infty}^{\tau}\psi(\vartheta)\d\vartheta = 1 -  \int_{\tau}^{\infty}\psi(\vartheta)\d\vartheta \end{equation}
and the height of the jump $h=b(1_{+0})-b(1_{-0})$; and
\item[\bf(2)] the matrix $\mathcal G(\tau,\theta,\Lambda,\epsilon)$ is given as a power series
\begin{equation}\label{eq:4.91}
 \mathcal G(\tau, \theta, \Lambda,\epsilon) = \mathrm I + \sum_{k=1}^\infty \Lambda^k \mathcal G_k(\tau,\theta,\epsilon)
\end{equation}
with coefficients $\mathcal G_k$ satisfying
\begin{equation}\label{eq:4.94}
 \|  \mathcal G_k(\tau,\theta,\epsilon)  \| \le \frac{C^k |\tau-\theta|^k}{k!}
\end{equation}
uniformly in $k$ and the occurring variables.
\end{enumerate}
\end{thm}

\begin{proof}
The representation for $\mathcal F$ follows by integrating the main diagonal part in equation (\ref{Transf to system in Zsing}). Using the explicit form of $\beta_0(\tau)$ from \eqref{eq:beta0eq} in combination with $\beta_{\epsilon}(\tau)=\beta_{0}(\tau)+\mathcal{O}(\epsilon)$, we obtain \eqref{eq:4.90}.

We make the ansatz
\begin{equation}
\mathcal E_{\rm sing}(\tau,\theta,\Lambda,\epsilon) =
\mathcal F(\tau,\theta,\epsilon) \mathcal G(\tau,\theta,\Lambda,\epsilon)
\end{equation}
for the fundamental solution to system \eqref{Transf to system in Zsing}. Then by construction
\begin{equation}\label{eq:4.93}
\partial_{\tau}\mathcal G(\tau,\theta,\Lambda,\epsilon)= \Lambda  \widetilde{\mathcal F}(\tau,\theta,\epsilon) 
\mathcal G(\tau,\theta,\Lambda,\epsilon)\qquad\text{with}\qquad \mathcal G(\theta,\theta,\Lambda,\epsilon)=\mathrm I,
\end{equation}
where the coefficient matrix satisfies
\begin{equation}\label{eq:tildeF_rep}
\begin{split}
\widetilde{\mathcal F}(\tau,\theta,\epsilon) &=
\mathcal F(\tau,\theta,\epsilon) \mathrm J \mathcal F(\theta,\tau,\epsilon)\\
&=
\begin{pmatrix}
0 &   \exp\left(-\int_{\theta}^{\tau} \beta_\epsilon(\vartheta)\d\vartheta \right) \\ 
-  \exp\left(-\int_{\tau}^{\theta} \beta_\epsilon(\vartheta)\d\vartheta \right) & 0
\end{pmatrix}
\\
&=\begin{pmatrix}
0 &   \exp\left(-\int_{\theta}^{\tau} \beta_0(\vartheta)\d\vartheta \right) \\ 
-  \exp\left(-\int_{\tau}^{\theta} \beta_0(\vartheta)\d\vartheta \right) & 0
\end{pmatrix} \left(1+ \mathcal O(\epsilon)\right). 
\end{split}
\end{equation}
In particular we obtain the uniform bound
\begin{equation}\label{eq:4.95}
   \| \widetilde{\mathcal F}(\tau,\theta,\epsilon)  \| \le C
\end{equation} 
independent of $\tau$, $\theta$ and $\epsilon$. Writing the solution to \eqref{eq:4.93} by the Peano--Baker series, we have for \eqref{eq:4.91} that
\begin{equation}
   \mathcal G(\tau,\theta,\Lambda,\epsilon) = \mathrm I + \sum_{k=1}^\infty \Lambda^k \int_\theta^{\tau}  
   \widetilde{\mathcal F}(\tau_1,\theta,\epsilon)  \int_\theta^{\tau_1} \widetilde{\mathcal F}(\tau_2,\theta,\epsilon) \int_\theta^{\tau_2}  \dots \int_\theta^{\tau_{k-1}}  \widetilde{\mathcal F}(\tau_k,\theta,\epsilon)  \d\tau_k\dots \d\tau_2\d\tau_1
\end{equation}
as a power series in $\Lambda$ with coefficients
\begin{equation}
 \mathcal G_k(\tau,\theta,\epsilon) = \int_\theta^{\tau}  
   \widetilde{\mathcal F}(\tau_1,\theta,\epsilon)  \int_\theta^{\tau_1} 
    \dots \int_\theta^{\tau_{k-1}}  \widetilde{\mathcal F}(\tau_k,\theta,\epsilon)  \d\tau_k\dots \d\tau_1.
\end{equation}
Combining this with \eqref{eq:4.95} concludes the proof.
\end{proof}

\begin{rem}
The asymptotic behaviour of the coefficient $\beta_\epsilon$ for $\epsilon\to0$ allows to simplify the formulas and to extract the asymptotic main contribution of the singular zone. Based on \eqref{Limiting behaviour of Beta} and \eqref{eq:beta0eq} we obtain
\begin{equation}  
\begin{split}
\exp\bigg(-\int_{\tau_{\Lambda_1}(\epsilon)}^{\tau_{\Lambda_2}(\epsilon)} \beta_\epsilon(\tau)\d\tau \bigg)
&= \exp\bigg( -\int_{\tau_{\Lambda_1}}^{\tau_{\Lambda_2}} \beta_0(\tau)\d\tau \bigg) \big(1+\mathcal O(\epsilon)\big)
\\&= \frac{b(1_{+0}) -h\int_{\tau_{\Lambda_1}}^{K}\psi(\theta)\d \theta}{b(1_{+0}) -h\int_{\tau_{\Lambda_2}}^{K}\psi(\theta)\d \theta}\big(1+\mathcal O(\epsilon)\big)
= \frac{b(1_{-0})}{b(1_{+0})} \big(1+\mathcal O(\epsilon)\big)
\end{split}
\end{equation}
with $\tau_{\Lambda_j}=\lim_{\epsilon\to0}\tau_{\Lambda_j}(\epsilon)$ and for all $\Lambda$ small enough to guarantee $\tau_{\Lambda_1}\le-K$ and $K\le \tau_{\Lambda_2}$.  Thus for all these $\Lambda$ we obtain
\begin{equation}
\mathcal F(\tau_{\Lambda_1},\tau_{\Lambda_2},\epsilon) =  \begin{pmatrix} 1 & 0\\0& \frac{b(1_{-0})}{b(1_{+0})} \end{pmatrix}\big(1+\mathcal O(\epsilon)\big).
\end{equation}
From \eqref{eq:4.91} and \eqref{eq:4.94} we conclude that $\mathcal G(\tau,\theta,\Lambda,\epsilon) - \mathrm I = \mathcal O(\Lambda)$ holds uniform with respect to $\epsilon$, $\tau$ and $\theta$ and hence 
\begin{equation}\label{eq:4.102}
 \mathcal E_{\rm sing}(\tau_{\Lambda_2},\tau_{\Lambda_1},\Lambda,\epsilon) =
 \begin{pmatrix} 1 & 0\\0& \frac{b(1_{-0})}{b(1_{+0})} \end{pmatrix} (1+\mathcal O(\epsilon) + \mathcal O(\Lambda))
\end{equation}
follows.
\end{rem}

\subsubsection{Limiting behaviour of the fundamental solution in the singular zone}
We want to describe the behaviour of the fundamental solution $\mathcal E_{\rm sing}(\tau_{\Lambda_{2}},\tau_{\Lambda_{1}},\Lambda,\epsilon)$ as $\epsilon\to0$ for fixed $\Lambda$. By \eqref{eq:4.90} we already know that the limit
\begin{equation}
\mathcal F(\tau,\theta,0) = \lim_{\epsilon\to0} \mathcal F(\tau,\theta,\epsilon)
\end{equation}
exists and satisfies
\begin{equation}\label{eq:limit_F}
\| \mathcal F(\tau,\theta,\epsilon) -\mathcal F(\tau,\theta,0) \| \le C\epsilon
\end{equation}
uniform in $\tau_{\Lambda_1} \le \theta < \tau \le \tau_{\Lambda_2}$. In the next step we consider the limiting behaviour of the power series $\mathcal G$ and in particular its coefficients
$\mathcal G_k$. 

\begin{lem}\label{prop:Estimate Q_sing}
The limit
\begin{equation}
\mathcal G_k(\tau,\theta,0) = \lim_{\epsilon\to0} \mathcal G_k(\tau,\theta,\epsilon)
\end{equation}
exists for all $\tau_{\Lambda_1} \le \theta < \tau \le \tau_{\Lambda_2}$ and satisfies 
\begin{equation}\label{eq:Gk_lim_Est}
   \|  \mathcal G_k(\tau,\theta,\epsilon) -  \mathcal G_k(\tau,\theta,0) \| \le C\epsilon.
\end{equation}
Furthermore,
\begin{equation}\label{eq:G1_rep}
 \mathcal G_1(\tau,\theta,0) = \int_\theta^\tau \begin{pmatrix}
0 &   \exp\left(\int_{\theta}^{\tau_1} \beta_0(\vartheta)\d\vartheta \right) \\ 
-  \exp\left(\int_{\tau_1}^{\theta} \beta_0(\vartheta)\d\vartheta \right) & 0
\end{pmatrix}\d\tau_1.
\end{equation}
\end{lem}

\begin{proof} The limiting behaviour of $\mathcal F$ implies that the limit
\begin{equation}
  \widetilde{\mathcal F}(\tau,\theta,0) = \lim_{\epsilon\to0} \widetilde{\mathcal F}(\tau,\theta,\epsilon)
\end{equation}
exists and is uniformly bounded with respect to $\tau_{\Lambda_1} \le \theta<\tau \le \tau_{\Lambda_2}$ and therefore the 
functions 
\begin{equation}
 \mathcal G_k(\tau,\theta,0) = \int_\theta^{\tau}  
   \widetilde{\mathcal F}(\tau_1,\theta,0)  \int_\theta^{\tau_1} 
    \dots \int_\theta^{\tau_{k-1}}  \widetilde{\mathcal F}(\tau_k,\theta,0)  \d\tau_k\dots \d\tau_1
\end{equation}
are good candidates to be considered for the limiting behaviour of $\mathcal G_k$. For $k=1$ the representation
\begin{equation}
   \mathcal G_1(\tau,\theta,0) = \int_\theta^\tau \widetilde{\mathcal F}(\tau_1,\theta,0) \d\tau_1
\end{equation}
corresponds directly to \eqref{eq:G1_rep} due to the formula \eqref{eq:tildeF_rep} for $\widetilde{\mathcal F}(\tau,\theta,0)$. Hence, using the analogue to \eqref{eq:limit_F} we obtain
\begin{equation}
    \|   \mathcal G_1(\tau,\theta,\epsilon)-   \mathcal G_1(\tau,\theta,0)\|
    \le  \int_\theta^\tau \|\widetilde{\mathcal F}(\tau_1,\theta,\epsilon)- \widetilde{\mathcal F}(\tau_1,\theta,0)\| \d\tau_1
    \le C |\tau-\theta| \epsilon 
\end{equation}
and using $|\tau-\theta|\le 2K'$ the first statement follows.

The estimate for $\mathcal G_k$ is obtained by telescoping the integral
\begin{equation}
\begin{split}
 \mathcal G_k(\tau,\theta,\epsilon)-   \mathcal G_k(\tau,\theta,0) =&
 \int_\theta^{\tau}  
  \big(  \widetilde{\mathcal F}(\tau_1,\theta,\epsilon) - \widetilde{\mathcal F}(\tau_1,\theta,0) \big)  \int_\theta^{\tau_1} 
   \widetilde{\mathcal F}(\tau_1,\theta,0)    \cdots \int_\theta^{\tau_{k-1}}  \widetilde{\mathcal F}(\tau_k,\theta,0)  \d\tau_k\dots \d\tau_1 \\
 &  +  \int_\theta^{\tau}  
  \big(  \widetilde{\mathcal F}(\tau_1,\theta,\epsilon) - \widetilde{\mathcal F}(\tau_1,\theta,0) \big)  \int_\theta^{\tau_1} 
\big( \widetilde{\mathcal F} (\tau_1,\theta,\epsilon)- \widetilde{\mathcal F}(\tau_1,\theta,0) \big) \\ &\qquad\times\int_\theta^{\tau_2} \widetilde{\mathcal F}(\tau_3,\theta,0)   \cdots \int_\theta^{\tau_{k-1}}  \widetilde{\mathcal F}(\tau_k,\theta,0)  \d\tau_k\dots \d\tau_1\\ & + \cdots,
\end{split}
\end{equation}
each term containing one difference more up to having $k$ differences as integrands. Note that this represents the difference $\mathcal G_k(\tau,\theta,\epsilon)-\mathcal G_k(\tau,\theta,0)$ in terms of the differences $ \widetilde{\mathcal F}(\tau_1,\theta,\epsilon) - \widetilde{\mathcal F}(\tau_1,\theta,0) $ and the form $\mathcal G_{k-\ell}(\tau_\ell,\theta,0)$ already estimated in the previous induction step. Hence
\begin{equation}
   \|  \mathcal G_k(\tau,\theta,\epsilon)-   \mathcal G_k(\tau,\theta,0) \|
   \lesssim \sum_{\ell=1}^k \epsilon^\ell   \lesssim \epsilon
\end{equation}
and the lemma is proved.
\end{proof}

\subsection{Bounded frequencies} We will give some remarks concerning estimates for the fundamental solution for $|\xi|\le N$.
Here it suffices to consider the system \eqref{Transf to system} in original form and to observe that its coefficient matrices have norm estimates $\|A(\xi)\| \lesssim |\xi|$ and $\|B(t,\epsilon)\| \lesssim 1+\Phi_\epsilon(t-1)$. 

Representing its solution directly by the Peano--Baker series yields 
\begin{equation}\label{eq:EboundSmallFreq}
\begin{split}
     \|\mathcal E(t,s,\xi,\epsilon) \| &\le \exp \bigg( C \int_s^t \big( |\xi| +1 + \Phi_\epsilon(\theta-1) \big) \d\theta \bigg)  \le \tilde C
\end{split}     
\end{equation}
using that $\int_0^2 \Phi_\epsilon(t-1) \d t$ is independent of $\epsilon$ and that both $|\xi|$ and $s,t$ are bounded.

\begin{rem}
Note that for dissipative problems the uniform boundedness of the fundamental solution follows already from the positivity of the coefficient of \eqref{Model} in front of $u_t$. For more general wave models this statement needs a proof and the above reasoning seems viable for this case too. 
\end{rem}

\subsection{Combining the bits}
We collect here the estimates obtained so far. As we are interested in the influence of the point singularity on the structure of the fundamental solution we consider  $t_1, t_2\in [0,2]$  with $t_{1} < 1 < t_{2}$ and look at the fundamental solution to \eqref{Transf to system} for fixed $\epsilon$ chosen sufficiently small. This is given by 
\begin{equation}
\mathcal E(t_{2},t_{1},\xi,\epsilon)=\mathcal E_{\rm hyp}(t_{2},t_{\xi_{2}}(\epsilon),\xi,\epsilon) T^{-1}(\epsilon)  \mathcal E_{\rm sing}(\tau_{\xi_{2}}(\epsilon),\tau_{\xi_{1}}(\epsilon),\epsilon|\xi|,\epsilon)   
 T(\epsilon) \mathcal E_{\rm hyp}(t_{\xi_{1}}(\epsilon),t_{1},\xi,\epsilon)
\end{equation}
with $T(\epsilon)$ the transformation matrix between the micro-energies used in the hyperbolic zone and in the singular 
zone, such as
\begin{equation}
T(\epsilon)  = \begin{pmatrix}\epsilon & 0 \\ 0 & \epsilon \end{pmatrix} = \epsilon\mathrm I.
\end{equation} 
Note that both of these matrices cancel each other and can therefore be neglected. 
As $\epsilon$ tends to $0$ we have $t_{\xi_{1}}\rightarrow 1_{-0}$ and $t_{\xi_{2}}\rightarrow 1_{+0}$. So using the estimates \eqref{eq:4.76} and \eqref{eq:4.102} we obtain for fixed $\xi$
\begin{equation}
\lim_{\epsilon\to0} \mathcal E(t_{2},t_{1},\xi,\epsilon) =
\mathcal E_{\rm hyp}(t_{2},1_{+0},\xi,0)
\begin{pmatrix}
1 & 0 \\ 
0 & H
\end{pmatrix}
\mathcal E_{\rm hyp}(1_{-0},t_{1},\xi,0),
\label{Fund sol behav}
\end{equation}
where $H=\frac{b(1_{-0})}{b(1_{+0})}$ is given in terms of the jump of $\log b$ at $t=1$.

\section{Results}

\subsection{Existence of very weak solutions}
Although in our model case the existence of very weak solutions was already established in \cite{MRT17}, we will show how to obtain this from the properties of the fundamental solution just constructed.

\begin{prop}\label{prop:A side result}
For $\epsilon\in \left(0,1\right]$, $0\le s <t \le 2$ and $|\xi|\ge N$ the fundamental solution to system \eqref{Transf to system}
 is uniformly bounded, i.e.
\begin{equation}
  \|  \mathcal E(t,s,\xi,\epsilon) \| \le C.
\end{equation}
\end{prop}

\begin{proof}
If $[s,t]\times \{(\xi,\epsilon)\}\subset Z_{\rm hyp}(N)$, the result follows directly from the construction in the hyperbolic zone. So it remains to consider only situations where the time interval intersects with the singular zone $Z_{\rm sing}(N)$.

We focus on the situation where $(s,\xi,\epsilon)\in Z_{\rm hyp}(N)$ and $(t,\xi,\epsilon)\in Z_{\rm sing}(N)$, i.e., $s<t_{\xi_1}(\epsilon) <1$ and $t_{\xi_1}(\epsilon) < t< t_{\xi_2}(\epsilon)$. Then the fundamental solution to system \eqref{Transf to system}
is given by
\begin{equation}
\mathcal E(t,s,\xi,\epsilon)=
T^{-1}(\epsilon)
 \mathcal E_{\rm sing}(\tau_{\xi_{2}},\tau_{\xi_{1}},\epsilon|\xi|,\epsilon)
  T(\epsilon)
 \mathcal E_{\rm hyp}(t_{\xi_{1}},t,\xi,\epsilon).
\end{equation}
As the factors $\epsilon^{-1}$ and $\epsilon$ arising from $T^{\pm1}(\epsilon)$ cancel out, it suffices to show the uniform boundedness of $\mathcal E_{\rm hyp}(t,s,\xi,\epsilon)$ for $s<t$ over the hyperbolic zone and that of $\mathcal E_{\rm sing}(\tau,\theta,\Lambda,\epsilon)$ for $\theta<\tau$ over the singular zone (in singular variables).

Hence, it remains to collect the already proved boundedness results. For $|\xi|\le N$ (i.e. for $\Lambda\le N\epsilon$) the uniform bound was shown in \eqref{eq:EboundSmallFreq}. For $|\xi|> N$ and within the hyperbolic zone the boundedness follows from the representation \eqref{eq:4.64} and the boundedness of each individual factor due to Theorem~\ref{thm:Fundamental sol in Zhyp}, while within the singular zone the representation in Theorem~\ref{thm:Fundamental sol in Zsing} gives a uniform bound on the fundamental solution based on the uniform boundedness of $\tau_{\Lambda_1}$ and $\tau_{\Lambda_2}$ with respect to both $\epsilon$ and $\Lambda$.
\end{proof}

In combination with the bound $\epsilon^{-1}+|\xi|$ for the coefficient matrix of \eqref{Transf to system} we conclude the bound
\begin{equation}
    \| \D_t^k \mathcal E(t,s,\xi,\epsilon) \| \le C_k \epsilon^{-k} |\xi|^k 
\end{equation}
uniform in $s<t$,  $\epsilon>0$ and $\xi\in\R^n$. 

\begin{cor}
Let the net $(u_\epsilon)_{\epsilon\in(0,1]}$ be a solution net to the Cauchy problem \eqref{Regularised pb} for initial data $u_0\in \mathrm H^1(\R^n)$ and $u_1\in\mathrm L^2(\R^n)$. Then the estimate 
\begin{equation}
      \| \partial_t^{1+k} u(t,\cdot) \|_{\mathrm H^{-k}(\R^n)} + \| \partial_t^k u(t,\cdot)\|_{\mathrm H^{1-k}(\R^n)}  \le C_k \epsilon^{-k}
\end{equation}
holds.
\end{cor}

\begin{rem}
Note that the negative power of $\epsilon$ only appears for the solution at and after the singularity $t=1$, and the estimates hold without  $\epsilon$ when $t<1$.
\end{rem}

\subsection{Exceptional propagation of singularities} Now we want to prove the exceptional propagation of singularities already hinted 
 by the numerical experiments from \cite{MRT17}. For this we consider the model problem in \emph{one} space dimension and use specially prepared initial data in the form of wave packets
\begin{equation}\label{eq:5:netU0}
\begin{split}
    u_0(x)   &=   \e^{ \i x \delta^{-1} \xi_0 } \chi(x),\\
     u_1(x) &= \partial_x u_0(x) =   \e^{ \i x\delta^{-1} \xi_0 }  \big( \i\xi_0 \delta^{-1} \chi(x) + \chi'(x) \big)
\end{split}
\end{equation}
parameterised by a fixed frequency $\xi_0\in\R\setminus\{0\}$ and for a smooth rapidly decaying function $\chi\in \mathcal S(\R)$ with sufficiently small Fourier support around the origin.  Applying a Fourier transform we see that 
\begin{equation}
\begin{split}
  |\xi| \widehat {u_0}(\xi) \pm \i   \widehat {u_1}(\xi) 
  &= |\xi| \widehat\chi(\xi- \delta^{-1} \xi_0) \mp  \xi  \widehat\chi(\xi-\delta^{-1} \xi_0),
  = \begin{cases} 0 , \qquad & \pm\xi>0, \\ \pm 2 \xi \widehat \chi(\xi-\delta^{-1} \xi_0), & \pm\xi<0. \end{cases}
\end{split}
\end{equation}
Without loss of generality we can assume that $\xi_0>0$ and $\supp \widehat\chi\subset [-\xi_0/2,\xi_0/2]$. Hence, for such initial data the initial datum $U_0(\xi)$ to \eqref{Transf to system} satisfies
\begin{equation}
   M^{-1} U_0(\xi,\epsilon) = {\sqrt2} \begin{pmatrix}
   0 \\ \xi \widehat \chi(\xi-\delta^{-1} \xi_0)
   \end{pmatrix}
\end{equation} 
for the diagonaliser $M$ from \eqref{eq:Mdef}.
Let now $t<1$. As $\mathcal E_0(t,s,\xi)$ is diagonal and $\mathcal Q(t,s,\xi,\epsilon)-\mathrm I$ as well as $N_1(t,s,\xi,\epsilon)-\mathrm I$ are both bounded by $|\xi|^{-1}$  uniformly in $\epsilon>0$ (small enough such that $(t,\xi,\epsilon)\in Z_{\rm hyp}(N)$) and $s\in[0,t]$ we obtain that 
\begin{equation}
 V(t,\xi,\epsilon) =\sqrt{ \frac{b_\epsilon(0)}{b_\epsilon(t)}} N_1(t,\xi,\epsilon) \mathcal E_0(t,0,\xi)   \mathcal Q(t,0,\xi,\epsilon)  N_1^{-1}(0,\xi,\epsilon) M^{-1} U_0(\xi,\epsilon) 
\end{equation}
is given by 
\begin{equation}
V(t,\xi,\epsilon) = \sqrt{ \frac{b(0)}{b(t)}}  {\sqrt2} \begin{pmatrix}
   0 \\  \e^{-\i t \xi}  \xi \widehat \chi(\xi-\delta^{-1} \xi_0)
   \end{pmatrix} 
   + \mathcal O(1),\qquad t<1,
\end{equation}
for fixed $t$ and with a uniformly bounded remainder independent of the choice of $\delta$.
This corresponds to a wave traveling to the right plus remainder terms with smaller norm. Note that the first term behaves like $
\delta^{-1}$ due to the support assumption made for $\widehat\chi$ and thus dominates the remainder term when choosing $\delta$ small enough. 

In the following, we consider $t>1$ and ask for the influence of the point singularity at time $1$ on the behaviour of our net of solutions. If $\epsilon >0$ is small enough such that $(t,\xi,\epsilon)\in Z_{\rm hyp}(N)$ the solution is represented by 
\begin{equation}
\begin{split}
V(t,\xi,\epsilon) &=
\sqrt{\frac{b_\epsilon(0)b_\epsilon(t_{\xi_2})}{b_\epsilon(t_{\xi_1}) b_\epsilon(t)}}
N_1(t,\xi,\epsilon) \mathcal E_0(t,t_{\xi_2},\xi)   \mathcal Q(t,t_{\xi_2},\xi,\epsilon)  N_1^{-1}(t_{\xi_2},\xi,\epsilon) \\
&\quad\times  M^{-1} T^{-1}(\epsilon) \mathcal E_{\rm sing}(\tau_{\xi_2},\tau_{\xi_1},\epsilon|\xi|,\epsilon) T(\epsilon) M \\
&\quad\times N_1(t_{\xi_1},\xi,\epsilon) \mathcal E_0(t_{\xi_1},0,\xi)   \mathcal Q(t_{\xi_1},0,\xi,\epsilon)  N_1^{-1}(0,\xi,\epsilon) M^{-1} U_0(\xi,\epsilon).
\end{split} 
\end{equation}
We again look at the main terms and estimates for remainders. In order to get the desired estimates we choose first the zone constant $N$ large enough to control non-diagonal terms appearing in the transformation matrices and in $\mathcal Q$. This yields based on the symbol estimate for $N_1(t,\xi,\epsilon)-\mathrm I$ and estimate \eqref{eq:Q-I_Est} for $\mathcal Q(t,s,\xi,\epsilon)-\mathrm I$
\begin{equation}
\begin{split}
   V(t,\xi,\epsilon) &= \sqrt{\frac{b(0)b(1_{+0})}{b(1_{-0})b(t)}}
     \frac1{\sqrt{2} }
     \begin{pmatrix}
     \e^{\i(t-1)\xi} & 0 \\ 0 & \e^{-\i (t-1) \xi} 
     \end{pmatrix}
     \begin{pmatrix} H+1 & H-1 \\ H-1 & H+1 \end{pmatrix}
     \begin{pmatrix}
   0 \\  \e^{-\i t \xi}  \xi \widehat \chi(\xi-\delta^{-1} \xi_0)
   \end{pmatrix} 
  \\&\quad  +  \mathcal O(\epsilon) + \mathcal O(\epsilon|\xi|) + \mathcal O(1/N)
\end{split}
\end{equation}
using in an essential way that the $T(\epsilon)$-terms cancel out, that $|t_{\xi_i}(\epsilon)-1|\le C\epsilon$ combined with
\begin{equation}
\begin{split}
  \|N_1(t,\xi,\epsilon)-\mathrm I\|+  \|N_1^{-1}(0,\xi,\epsilon) - \mathrm I \|  &\le C |\xi|^{-1},\\
  \|\mathcal Q(t,t_{\xi_2},\xi,\epsilon)-\mathrm I\| +  \|\mathcal Q(t_{\xi_1},0,\xi,\epsilon) - \mathrm I\| &\le C/N, \\
 \|N_1^{-1}(t,t_{\xi_2},\xi,\epsilon)-\mathrm I\| + \|N_1(t_{\xi_1},\xi,\epsilon) - \mathrm I \|  &\le C /N 
\end{split}
\end{equation}
due to \eqref{eq:N1}, \eqref{eq:Q-I_Est}  and \eqref{eq:txi-def} and that 
\begin{equation}
\begin{split}
   M^{-1}  \mathcal E_{\rm sing}(\tau_{\xi_2},\tau_{\xi_1},\epsilon|\xi|,\epsilon) M
   & =  \frac12 \begin{pmatrix} 1 & 1 \\ -1 & 1 \end{pmatrix}
   \begin{pmatrix} 1 & 0 \\ 0 & H \end{pmatrix} 
   \begin{pmatrix} 1 & -1 \\ 1 & 1 \end{pmatrix}
   + \mathcal O(\epsilon) + \mathcal O(\epsilon|\xi|)
   \\& =  \frac12 \begin{pmatrix} H+1 & H-1 \\ H-1 & H+1 \end{pmatrix}
    + \mathcal O(\epsilon) + \mathcal O(\epsilon|\xi|)
\end{split}
\end{equation}
due to \eqref{eq:4.102}  with $H  =\frac{b(1_{-0})}{b(1_{+0})}\in(0,1]$. As for our net of initial data $|\xi|\sim \delta^{-1}$, the second remainder term is of order $\epsilon \delta^{-1}$ and thus negligible for $\epsilon$ small enough and $\delta$ fixed.

To recover the solution $u(t,x)$, we have to multiply by the matrix $M$ and apply the inverse Fourier transform. Thus we obtain the following theorem.

\begin{thm}
The very weak solution corresponding to the net of initial date \eqref{eq:5:netU0} is described (up to terms small compared to the solution itself) 
\begin{itemize}
\item by a wave travelling to the right for $t<1$; and
\item by two waves travelling to the left and to the right for all $t>1$.
\end{itemize}
\end{thm}

\begin{rem}
The partial reflection of rays at the singularity is characterised by the matrix
\begin{equation}
   \frac1{2} \begin{pmatrix} H+1 & H-1 \\ H-1 & H+1 \end{pmatrix}
\end{equation}
in terms of the jump of $\log b$ at $t=1$. 

Thus, if the coefficient $b$ has no jump and therefore $H=1$ this matrix becomes the identity matrix and for $t>1$ only one wave propagates to the right. Hence, no reflected wave occurs.

If $b$ has a jump we can compare the amplitude of both travelling waves. For this we fix a sufficiently small $\delta>0$ and write down the main terms of the travelling wave as
\begin{equation}
   u(t,x) = \sqrt{\frac{b(0)}{b(t)}} \mathbf u(x-t),\qquad 0<t<1
\end{equation}
and 
\begin{equation}
   u(t,x) =  \frac{H+1}{2\sqrt H}\sqrt{\frac{b(0)}{b(t)}} \mathbf u(x-t)   +   \frac{H-1}{2\sqrt H} \sqrt{\frac{b(0)}{b(t)}} \mathbf u(x-2+t),\qquad t>1.
\end{equation}
The first term corresponds to a wave continuing in the same direction but with amplitude multiplied by $\frac{H+1}{2\sqrt H}$, while the second term gives the reflected part with amplitude multiplied by $\frac{H-1}{2\sqrt H}$.  
\end{rem}

\begin{rem}
The related wave model
\begin{equation}
  u_{tt} - \Delta u + \delta_1(t) u_t = 0
\end{equation}
with coefficient given by the Delta distribution supported in $t=1$ appears almost as a special case of treatment here in the paper. For the choice of $b(t) = 1/2$ for $0\le t<1$ and $b(t)=3/2$ for $1<t\le 2$ we obtain a closely related net of coefficients leading to $H=1/3$ and a resulted transfer matrix at the singularity.

The true consideration of the above equation can be done on lines similar to the treatment provided here in the paper. This would lead to the (related) transfer matrix
\begin{equation} \frac1{2\e} \begin{pmatrix}  1+ \e & 1-\e \\ 1-\e & 1+\e \end{pmatrix} \end{equation} .
\end{rem}

\begin{rem}
The arguments presented in this section for the case of one space dimension applies in a similar way to higher dimensions. The main reflected wave travels in the opposite direction to the main one, and lower order terms could propagate along cones emanating from the point of interaction of singularities.
\end{rem}

\section{Concluding remarks}

We will conclude this article with some comments on the tools and techniques developed so far and mention some open problems and challenges.

\begin{enumerate}
\item The symbol classes used in the treatment were adapted to one point singularity at $t=1$. This can clearly be extended to treat point singularities at a finite number of times.
\item Using the same symbol classes one can treat other wave models with time-dependent coefficients having point singularities of suitable strength. This corresponds to the models proposed in \cite{ART19}  and will be considered in details in a forthcoming paper. 
\item A related problem are singular wave models with singularities depending on space and time. Here an adapted version of a full $\epsilon$-dependent pseudo-differential calculus has to be used in order to describe the propagation of singularities for very weak solutions. However, this is a much harder problem and the description of a local scattering process of waves (and thus wave front sets) of very weak solutions at such singularities remains challenging. 
\end{enumerate}

\subsection*{Acknowledgements} Mohammed Sebih visited the Department of Mathematics of the University of Stuttgart  supported by an Algerian Scholarship P.N.E. 2018/2019. He thanks University of Stuttgart for providing excellent working conditions.

\bibliographystyle{siam}
\bibliography{SingDiss}

\end{document}